\documentclass[preprint,10pt]{elsarticle}

\usepackage{amsmath,amssymb,amsfonts}
\usepackage{amsthm}

\usepackage{graphicx}
\usepackage{booktabs}
\usepackage{float}
\usepackage{tabularx}
\usepackage{array} 

\newcolumntype{Y}{>{\centering\arraybackslash}X}

\usepackage{xcolor}
\usepackage[protrusion=true,expansion=false]{microtype}
\usepackage{tikz}
\usetikzlibrary{arrows,shapes,patterns,positioning,calc,arrows.meta}

\usepackage{hyperref}
\hypersetup{
  colorlinks=true,
  citecolor=blue,
  urlcolor=blue,
  linkcolor=blue
}

\theoremstyle{plain}
\newtheorem{theorem}{Theorem}[section]

\newtheorem{proposition}[theorem]{Proposition}

\theoremstyle{definition}

\theoremstyle{remark}
\newtheorem{remark}{Remark}[section]

\journal{Mathematical Biosciences}

\begin{document}

\begin{frontmatter}

\title{Logistic Gene Regulatory Networks: A Modelling Framework Beyond Hill Functions}

\author[inst1]{Ismail Belgacem}
\ead{ismail.belgacem.81@gmail.com}

\address{Mezaourou, Ghazaouet 13421, Tlemcen, Algeria}

\begin{abstract}
Boolean network models are a widely used framework for describing gene
regulatory networks across many biological systems, from the mammalian cell
cycle to cancer-signalling and developmental decision circuits. Extracting
quantitative dynamics from such a model---its attractors, basins, and
transition timing---requires translating its logical update rules into a
continuous system of ordinary differential equations, and the sigmoidal kernel
chosen for this translation is a modelling decision with direct biological
consequences. The near-universal choice, the Hill function, sets production to
exactly zero when an activator is absent; yet genes are never fully silent, so
this idealisation introduces a spurious absorbing off-state with no biological
counterpart. We develop a general product-of-logistics framework in which
increasing logistic functions represent activation, decreasing logistic
functions represent repression, and a recursive De~Morgan product formula
translates an arbitrary Boolean rule---conjunctions, disjunctions, and
negations---into a continuous regulatory function. The translation is
automatic, confines every regulatory function to the unit interval, and retains
a strictly positive basal rate. Our central result is a recovery theorem:
every steady state of the Boolean network reappears, for sufficiently steep
regulatory response, as an exponentially stable equilibrium of the continuous
model, with the discrete labels $0$ and $1$ realised as basal and saturated
concentrations, so the translation provably refines, rather than distorts, the
original Boolean analysis. We establish the
analytical foundations the framework requires---global well-posedness, forward
invariance, an explicit Lipschitz constant, and, for the two canonical
two-gene motifs, both the global asymptotic stability of the negative-feedback
oscillator and a closed-form bistability threshold for the genetic toggle
switch---and we show that every regulator threshold remains a positive,
experimentally measurable concentration, unlike weighted-sum logistic
formulations that place repressor thresholds at biologically meaningless
negative values. The eleven-gene Traynard mammalian cell-cycle network is
translated automatically and integrated: in the proliferative regime its
trajectories settle onto a sustained limit cycle that reproduces the cyclic
attractor of the underlying Boolean model. Because the translation is purely
structural, the same procedure applies without modification to existing Boolean
models of cancer signalling and developmental transitions, and the framework
further supports exact feedback linearisation for control design.
\end{abstract}


\end{frontmatter}

\section{Introduction}
\label{sec:introduction}

Gene regulatory networks (GRNs) constitute the control architecture of living
cells, orchestrating the spatiotemporal patterns of gene expression that
underlie development, homeostasis, and adaptation to environmental change.
For many specific biological systems, the structure of these networks is
captured by Boolean models, in which each gene carries a logical update rule
over its regulators. This discrete formalism has produced validated,
predictive descriptions of processes as diverse as the mammalian cell
cycle~\citep{Traynard2016}, breast-cancer signalling~\citep{biane2018causal},
the ERBB-regulated G1/S transition~\citep{Sahin2009}, S-phase entry and
senescence~\citep{Verlingue2016}, the epithelial-to-mesenchymal
transition~\citep{Cohen2015}, the pro-inflammatory microenvironment of acute
lymphoblastic leukaemia~\citep{Enciso2016}, and developmental
patterning~\citep{albert2003topology}. A Boolean model identifies the
attractors of a network, but it cannot by itself resolve graded expression
levels, transition timing, parameter sensitivities, or responses to
continuously varying inputs; obtaining these requires translating the logical
update rules into a continuous system of ordinary differential equations, a
step now routine in quantitative systems biology. The mathematical analysis of
such continuous models---the global stability of gene-expression and enzymatic
reaction dynamics~\citep{belgacem2012global,belgacem2013enzymatic,belgacem2014concave},
model reduction for coupled transcription--translation
systems~\citep{belgacem2013stability,belgacem2013analysis,belgacem2014coupled,belgacem2014stability,belgacem2018reduction},
and the control of genetic feedback
loops~\citep{belgacem2020control,chambon2020qualitative}---provides the
methodological foundation on which the present logistic framework builds. The sigmoidal kernel chosen
for that translation is not a neutral technical detail: it determines whether
the continuous model faithfully refines the Boolean picture or silently
distorts it, and the same choice governs the emergent dynamical
behaviours---oscillations, bistability, multistability, and
chaos~\citep{farcot2019chaos,belgacem2025glass}---that the continuous model is
built to explain. The near-universal choice of kernel, both for
Boolean-to-ODE translation and for sigmoidal GRN modelling more broadly, is
the Hill function, a formulation introduced over a century ago to describe
cooperative ligand binding.
Its activation form $h^+(x,\theta,n) = x^n/(x^n+\theta^n)$ and repression
counterpart $h^-(x,\theta,n) = \theta^n/(x^n+\theta^n)$ are intuitive,
mechanistically grounded, and have been applied successfully to models of the
lambda phage lysis--lysogeny decision, the \textit{lac} and \textit{gal}
operons in \textit{Escherichia coli}, developmental patterning in
\textit{Drosophila}, and mammalian cell-cycle control, among many others.
The Hill coefficient $n$ quantifies cooperativity, $\theta$ denotes the
half-maximal concentration, both carry clear physical interpretations, and the
sigmoidal shape faithfully captures the switch-like regulatory transitions
characteristic of biological decision-making.

Yet the Hill function carries structural liabilities, and the first is a
matter of biological fidelity. Experimental studies across bacterial operons,
eukaryotic promoters, and synthetic circuits consistently show that genes are
never fully silent: even under strong repression, a small basal transcription
persists~\citep{becskei2000engineering,lipshtat2006genetic,weickert1993galactose}.
The Hill function cannot represent this, because $h^+(0,\theta,n)=0$ for every
$n>0$, so any gene whose activator is momentarily absent is assigned a
production rate of exactly zero. In a bistable circuit this manufactures an
absorbing off-state with no biological counterpart, and in a Boolean-derived
network it freezes, from the first integration step, every target whose
activator starts below threshold---so the continuous model can settle into
attractors that the underlying biology does not possess. The remaining two
liabilities are more technical. The second is
analytical: when the cooperativity exponent $n$ takes a non-integer
value---as it routinely does when fitting experimental dose-response data,
yielding values such as $n\approx1.39$, $2.73$, or
$3.52$~\citep{gottschalk2005five,reeve2013pharmacodynamic,santillan2008use}---the
power-law form $x^n$ is only $C^{\lfloor n\rfloor}$-smooth at the origin, so
higher-order analytical tools such as centre-manifold reduction and normal-form
analysis are unavailable, and, although the rational form
$x^n/(x^n+\theta^n)$ does invert in closed form, it does so only through a
fractional power $\theta\,(y/(1-y))^{1/n}$ that carries the same root
non-smoothness into any feedback-linearising control law. The third is
parametric: the Hill maximum slope $n/(4\theta)$ entangles cooperativity and
threshold, so the two cannot be tuned independently.

The logistic function resolves all three issues simultaneously.
Its activation form $f^+(x,\theta,\lambda) = 1/(1+e^{-\lambda(x-\theta)})$ and
repression form $f^-(x,\theta,\lambda) = 1/(1+e^{\lambda(x-\theta)})$ are
globally $C^\infty$, real-valued for all arguments including negative ones, and
strictly positive at zero concentration for all finite $\lambda$ and $\theta$.
The self-referential derivative identity $f' = \lambda f(1-f)$ reduces Jacobian
entries to products of function values, eliminating fractional exponents from
stability analysis entirely.
The closed-form logit inverse $f^{-1}(y) = \theta + \lambda^{-1}\ln(y/(1-y))$
enables exact feedback linearisation.
The parameters $\theta$ and $\lambda$ are fully decoupled: the threshold can be
repositioned without altering the transition slope, and vice versa; both map
directly to biologically measurable quantities---$\theta$ to dissociation
constants or half-maximal effective concentrations, and $\lambda$ to effective
cooperativity.
Most importantly, the non-zero output at $x=0$ means that basal expression is
built into the function's shape rather than appended to it.

Logistic functions have a long and productive history in statistics, machine
learning, and neural-network approximation theory, and have recently appeared
in GRN modelling through the work of~\citet{samuilik2022mathematical}, who used
a single increasing sigmoid for all regulatory interactions, with signed
weights encoding direction.
While this provides a unified formalism with attractive mathematical
properties, it forces both activation and repression through the same
functional form. For repression in particular, this leads to critical points
at biologically meaningless negative concentrations and a systematic failure to
approach unity under unrepressed conditions---pathologies that arise not from
biological necessity but from the modelling choice itself.
The present paper takes a fundamentally different approach: by deploying
increasing logistic functions for activation and decreasing logistic functions
for repression, each precisely where it is biologically appropriate, we
preserve the distinct sigmoidal dynamics of each regulatory mode.
We show that increasing and decreasing sigmoids can be combined in a product
that naturally encodes AND combinatorial logic, and we compare this formulation
in detail with the Samuilik weighted-sum alternative.

This paper develops the modelling framework and establishes its core
analytical properties. After introducing the logistic formulation
(Section~\ref{sec:modeling}), we analyse the two-gene negative-feedback
oscillator and prove, via Jacobian analysis combined with Bendixson's negative
criterion and the Poincar\'e--Bendixson theorem, that it is globally
asymptotically stable and cannot undergo a Hopf bifurcation
(Theorem~\ref{thm:no_hopf}), so that sustained limit cycles require explicit
time delays. The complementary mutual-repression motif, the genetic toggle
switch of~\citet{gardner2000construction}, is analysed in the same closed form
(Section~\ref{ex:toggle}): a single inequality, $\rho\lambda>4$ in the
symmetric case, separates monostable from bistable behaviour, and the
transition is a supercritical pitchfork. We then formulate the general
multi-gene system as a
product-of-logistics map; Proposition~\ref{prop:demorgan} establishes the three
structural properties of the recursive De~Morgan map---range, Boolean
consistency, and De~Morgan duality---and Proposition~\ref{prop:wellposedness}
establishes global well-posedness, forward invariance, and an explicit global
Lipschitz constant, and Theorem~\ref{thm:boolean_recovery} proves that every
steady state of the Boolean network is recovered, for sufficiently steep
regulatory response, as an exponentially stable equilibrium of the continuous
system. The $11$-gene Traynard mammalian cell-cycle
network~\citep{Traynard2016} is translated automatically into a continuous ODE
system that integrates without warnings and, in the proliferative regime,
settles onto a sustained limit cycle reproducing the cyclic attractor of the
Boolean model; because the De~Morgan translation is purely structural, the same
automatic procedure applies without modification to the other documented Boolean
GRN models cited above, from cancer-signalling networks to developmental
decision circuits. Section~\ref{sec:weight_equivalence}
proves that incorporating explicit interaction weights changes nothing after a
parameter rescaling, and Section~\ref{sec:comparison_literature} compares the
product-of-logistics formulation with the Samuilik weighted-sum alternative
across the AND, OR, and NOR logic gates. Finally, Section~\ref{sec:control}
turns to control design: it shows that the always-positive logistic production
rate removes the controllability gaps that Hill-based models exhibit at zero
concentration, and that the closed-form logit inverse supports an exact
feedback-linearisation construction with provable exponential tracking
(Proposition~\ref{prop:tracking}).

A companion paper~\citep{belgacem2026numerical} builds on the framework
established here to address the biological and computational consequences of
the logistic choice: the prevention of expression shutdown in low-expression
regimes, and the numerical reliability of Boolean-derived ODE integration
relative to Hill functions with non-integer exponents. Taken together, the two papers establish
logistic functions not as a minor variation on an established theme but as a
principled, analytically tractable, and biologically faithful foundation for
GRN modelling.

\section{The Logistic Modelling Framework}
\label{sec:modeling}

Gene regulatory networks exhibit inherently nonlinear dynamics across multiple coupled components, characterised by sigmoidal activation and repression functions, feedback loops generating bistability and oscillations, threshold-activated switches enabling binary cellular decisions, and saturating responses reflecting finite molecular resources~\citep{belgacem2020control,chambon2020qualitative,farcot2019chaos,belgacem2025glass}. Realistic models routinely involve hundreds to thousands of interacting genes whose dynamics span molecular binding events on millisecond timescales, cellular differentiation processes unfolding over hours to days, and population-level dynamics evolving across generations, while molecular noise arising from low copy numbers (typically $10$--$1000$ transcription factor molecules per cell) and cell-to-cell variability in isogenic populations further complicates the picture. This section develops the logistic alternative to Hill functions and demonstrates that it resolves the structural pathologies of the latter while preserving its biological content.

\subsection{Limitations of Hill-Function-Based Models}
\label{sec:hill_pathologies}

For decades, Hill functions have dominated biological modelling, valued for
their mechanistic foundation in equilibrium binding theory and their ability
to encode cooperative molecular binding through fractional
exponents~\citep{bernot2012modeling,polynikis2009comparing,bottani2017hill,kim2011stochastic}.
The increasing form $h^+(x, \theta, n) = \frac{x^n}{\theta^n + x^n}$ and the
decreasing repression form $h^-(x, \theta, n) = \frac{\theta^n}{\theta^n + x^n}$
have been applied successfully across bacterial gene circuits, mammalian
signalling pathways, metabolic regulation, and synthetic biology
applications~\citep{belgacem2020control,bernot2012modeling,kim2011stochastic,polynikis2009comparing}.
The Hill coefficient $n$ admits both phenomenological and mechanistic
interpretations: operationally, experimentalists extract $n$ from sigmoidal
dose-response fits as a measure of response steepness and local input-output
sensitivity, while mechanistically $n$ approximates the number of interacting
binding sites or the degree of cooperativity.

The \textit{lac} operon illustrates both the utility and the limits of this
picture. Simple LacI repression at a single operator is described by $n=1$;
the observed $>1000$-fold repression arises from DNA looping between the main
and auxiliary operators, captured by a distinct thermodynamic expression rather
than a large Hill coefficient~\citep{ozbudak2004multistability,madar2011negative,oehler1994quality}.
More generally, any Hill coefficient $n>1$ from dose-response fitting is a
phenomenological summary of the overall network response rather than a direct
mechanistic count of binding steps. This understanding informs synthetic biology
applications, where engineered toggle switches and oscillators in \textit{E.~coli}
and yeast rely on cooperative binding (typically $n=2$--$4$) to achieve bistable
memory devices and biosensors~\citep{gardner2000construction,elowitz2000synthetic}.

However, this classical choice carries substantial hidden costs that severely
limit mathematical analysis and computational
efficiency~\citep{belgacem2025exploring}. When Hill models are fitted to
experimental data, the resulting Hill coefficients frequently assume non-integer
values, reflecting incomplete cooperativity, heterogeneous binding-site
occupancy, or complex allosteric
mechanisms~\citep{gottschalk2005five,reeve2013pharmacodynamic,santillan2008use,abeliovich2005empirical}.
This transition from integer to non-integer $n$ precipitates a loss of
mathematical structure. Hill functions are only $C^{\lfloor n \rfloor}$-smooth:
for $n \in (k, k+1)$, derivatives of order greater than $k$ diverge at the
origin, restricting the applicability of centre manifold theory, normal form
analysis, and higher-order perturbation methods, and causing step-size
inflation and stalling in adaptive ODE solvers.

The derivative of the activation Hill function involves fractional powers
that resist symbolic manipulation: $\frac{dh^+}{dx} = \frac{n\theta^n
x^{n-1}}{(\theta^n + x^n)^2}$. For non-integer $n$, computing $x^n$ requires
the transcendental composition $e^{n \ln x}$, which accumulates floating-point
error near zero and becomes complex-valued whenever any trajectory
component overshoots to a negative concentration, as every adaptive ODE solver
will produce from rounding errors alone. The inverse $\theta\,(y/(1-y))^{1/n}$
is itself available in closed form, but only as a fractional power, so any
feedback-linearising control law inherits the same root evaluations---non-smooth
at the origin and complex-valued for negative arguments---rather than the
elementary inverse the logistic provides. A rigorous mathematical analysis of these
limitations is provided in~\citep{belgacem2025exploring}.

\subsection{The Logistic Formulation}
\label{subsec:logistic_formulation}

Logistic and Hill functions both generate smooth sigmoidal response curves with tunable steepness and threshold, but the logistic form achieves this biological fidelity without the analytical pathologies of Hill: no fractional exponents, no power-law singularities at the boundary, and no ill-conditioned numerical behaviour~\citep{belgacem2025exploring}. We model gene activation by the increasing logistic
\begin{equation}
    f^+(x, \theta, \lambda) = \frac{1}{1 + e^{-\lambda (x - \theta)}},
    \label{eq:logistic_activation}
\end{equation}
and gene repression by the decreasing logistic
\begin{equation}
    f^-(x, \theta, \lambda) = \frac{1}{1 + e^{\lambda (x - \theta)}}
    = \frac{1}{1 + e^{-\lambda (\theta - x)}},
    \label{eq:logistic_repression}
\end{equation}
each deployed precisely where it is biologically appropriate. The full mathematical analysis of these functions, including global $C^\infty$ smoothness, real-valuedness for all arguments, and the closed-form logit inverse, is established in~\citep{belgacem2025exploring}; the present paper assumes these properties and develops their consequences for biological modelling and control.

Two structural features of the logistic form are central to what follows. First, the self-referential derivative identity
\[
\frac{\partial f^+}{\partial x} = \lambda f^+(1 - f^+), \qquad
\frac{\partial f^-}{\partial x} = -\lambda f^-(1 - f^-)
\]
reduces every Jacobian entry to a product of function values and so eliminates the fractional power evaluations that make Hill Jacobians ill-conditioned near zero. Second, the closed-form inverse
\[
(f^+)^{-1}(y) = \theta + \tfrac{1}{\lambda}\ln\!\left(\tfrac{y}{1-y}\right)
\]
enables exact feedback linearisation through an elementary, fractional-power-free expression; the Hill form does invert, to $\theta\,(y/(1-y))^{1/n}$, but only via a fractional power whose non-smoothness at the origin the logit avoids.

To compare quantitatively with a Hill function of cooperativity~$n$, the steepness parameter is matched by equating slopes at the half-maximal threshold: at $x = \theta$, the Hill slope is $n/(4\theta)$ and the logistic slope reaches its maximum $\lambda/4$, so the rule $\lambda = n/\theta$ makes the two functions locally identical at $x = \theta$ (in value and slope), with global agreement extending only within a neighbourhood of the threshold; away from $x = \theta$ the two responses diverge, with the Hill function evolving along a logarithmic scale and the logistic along a linear scale~\citep{belgacem2025exploring}. The analogous matching applies to the decreasing forms. Crucially, the logistic slope at threshold $\lambda/4$---which is also its global maximum---depends on $\lambda$ \emph{alone}, whereas the Hill slope at threshold $n/(4\theta)$ entangles cooperativity and threshold, so the logistic form alone permits independent tuning of decision threshold and response sensitivity in circuit design and parameter estimation. Table~\ref{table_1} summarises this comparison.

\begin{table}[h!]
\centering
\caption{Principal mathematical distinctions between Hill and logistic GRN
models. Both functions take the same absolute concentration $x$ as input; the
differences are structural, not a matter of ``absolute'' versus ``relative''
scaling.}
\label{table_1}
\small
\begin{tabularx}{\linewidth}{@{}l Y Y@{}}
\toprule
Property & Hill $h^+(x,\theta,n)$ & Logistic $f^+(x,\theta,\lambda)$ \\
\midrule
Value at $x=0$ & $0$ (absorbing state) & $1/(1+e^{\lambda\theta}) > 0$ (basal expression) \\
Smoothness & $C^{\lfloor n \rfloor}$ at $x=0$ (non-integer $n$) & $C^\infty$ everywhere \\
Lipschitz constant & $+\infty$ near $x=0$ if $n<1$; finite for $n\ge1$ & $\lambda/4 < \infty$ globally \\
Sensitivity structure & Fold-change (multiplicative in $x/\theta$) & Additive deviation from $\theta$ \\
Slope at threshold & $n/(4\theta)$ (couples $n$ and $\theta$; below the maximum, which sits at the inflection $x<\theta$) & $\lambda/4$ (independent of $\theta$; also the global maximum) \\
Closed-form inverse & Yes, but a fractional power: $\theta\,(y/(1-y))^{1/n}$ & Yes, elementary logit: $\theta + \lambda^{-1}\ln(y/(1-y))$ \\
Closed-form integral & No elementary form (hypergeometric for non-integer $n$) & Yes: $\lambda^{-1}\ln(1+e^{\lambda(x-\theta)})$ \\
Basal expression & Requires ad hoc offset $\varepsilon$ & Built in; controlled by $\lambda\theta$ \\
Parameter coupling & Steepness $n/\theta$ tied to threshold & $\lambda$ and $\theta$ independent \\
\bottomrule
\end{tabularx}
\end{table}

\subsection{Non-Cooperative Gene Regulatory Networks}
\label{subsec:noncoop}

In the simplest regulatory scenarios, where each gene responds to a single regulator without cooperative binding effects, the logistic framework takes a particularly streamlined form. Here, gene $i$ is regulated by the expression level of gene $x_j$ through a single logistic function, with the regulatory direction (activation or repression) encoded by a sign parameter. The general dynamical model is:
\begin{equation}
\dot{x}_i = \kappa_i \frac{1}{1 + e^{-\sigma_i \lambda (x_j - \theta_i)}} - \gamma_i x_i, \quad i = 1, \ldots, N,
\label{eq:general_logistic_model}
\end{equation}
where $x_i$ denotes the expression level (for instance, the protein concentration) of gene~$i$; $x_j$ is the concentration of the regulatory species (with $j = i-1$ in a sequential cascade, or any other index in the network); $\kappa_i > 0$ and $\gamma_i > 0$ are the maximal production rate and degradation rate, respectively; $\lambda > 0$ controls the steepness of the regulatory response; $\theta_i > 0$ is the regulatory threshold; and $\sigma_i \in \{+1, -1\}$ encodes the regulatory sign, with $\sigma_i = +1$ for activation and $\sigma_i = -1$ for repression.

This formulation embeds both activation and repression within a unified structure, with the sole distinction being the sign $\sigma_i$ that inverts the exponential argument.

\section{Two-Gene Regulatory Motifs}
\label{sec:two_gene}

\subsection{The Two-Gene Oscillator: An Illustrative Example}
\label{ex:oscillator}

As a canonical illustration of the logistic framework, consider one of the most fundamental motifs in biological networks: the two-gene negative-feedback oscillator, which underlies phenomena ranging from circadian rhythms to cell-cycle progression. In its simplest form, the first gene activates the second, which in turn represses the first, creating a cyclical pattern of expression (Fig.~\ref{fig:oscillator}).

\begin{figure}[t]
\begin{center}
\rotatebox{-90}{
\begin{tikzpicture}[font=\small, >=stealth]
    \node[name=A, text centered, shape = circle, shading=ball, color = yellow!20]  at (0,0) {$ \rotatebox{90}{A}$};
    \node[name=B, text centered, shape = circle, shading=ball, ball color = yellow!20]  at (0,2) {$\rotatebox{90}{B}$};
    \draw [-|, color = red, shorten <=1pt, shorten >=1pt, very thick] (B) -- (-0.68,2) -- (-0.68,0) -- (A);
    \draw[->, shorten <=1pt, blue, shorten >=1pt, very thick] (A.75) -- (B.-75);
\end{tikzpicture}
}
\caption{Architecture of a two-gene negative feedback loop. Gene A  activates gene B (blue arrow), while gene B represses gene A (red bar).}
\label{fig:oscillator}
\end{center}
\end{figure}
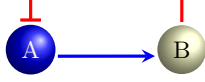

Deploying our logistic formulation, the dynamics of this system are governed by the coupled ordinary differential equations:
\begin{equation}
\begin{aligned}
\dot{x}_1 &= \kappa_1\, f^-(x_2,\theta_2,\lambda) - \gamma_1 x_1
        \;=\; \kappa_1\, \frac{1}{1 + e^{\lambda (x_2 - \theta_2)}} - \gamma_1 x_1, \\[6pt]
\dot{x}_2 &= \kappa_2\, f^+(x_1,\theta_1,\lambda) - \gamma_2 x_2
        \;=\; \kappa_2\, \frac{1}{1 + e^{-\lambda (x_1 - \theta_1)}} - \gamma_2 x_2,
\end{aligned}
\label{eq:oscillator}
\end{equation}
where $\kappa_1, \kappa_2 > 0$ are the maximal production rates, $\gamma_1, \gamma_2 > 0$ are the degradation rates, $\lambda > 0$ determines the steepness of both regulatory responses, and $\theta_1, \theta_2 > 0$ are the respective threshold concentrations. Gene~1 is repressed by gene~2 via the decreasing logistic $f^-$ (which decreases monotonically as $x_2$ rises), while gene~2 is activated by gene~1 via the increasing logistic $f^+$.

Numerical simulations of this system with parameter values $\lambda = 3$, $\kappa_1 = 3$, $\gamma_1 = 0.25$, $\kappa_2 = 4$, $\gamma_2 = 0.5$, $\theta_1 = 4$, $\theta_2 = 3$, starting from initial conditions $x_1(0) = x_2(0) = 1$, are depicted in Fig.~\ref{fig:Oscillateur_original}. The system exhibits damped oscillations approaching the equilibrium approximately at $(x_1^*, x_2^*) \approx (3.87, 3.25)$.

\begin{figure}[t]
    \centering
    \includegraphics[scale=0.2]{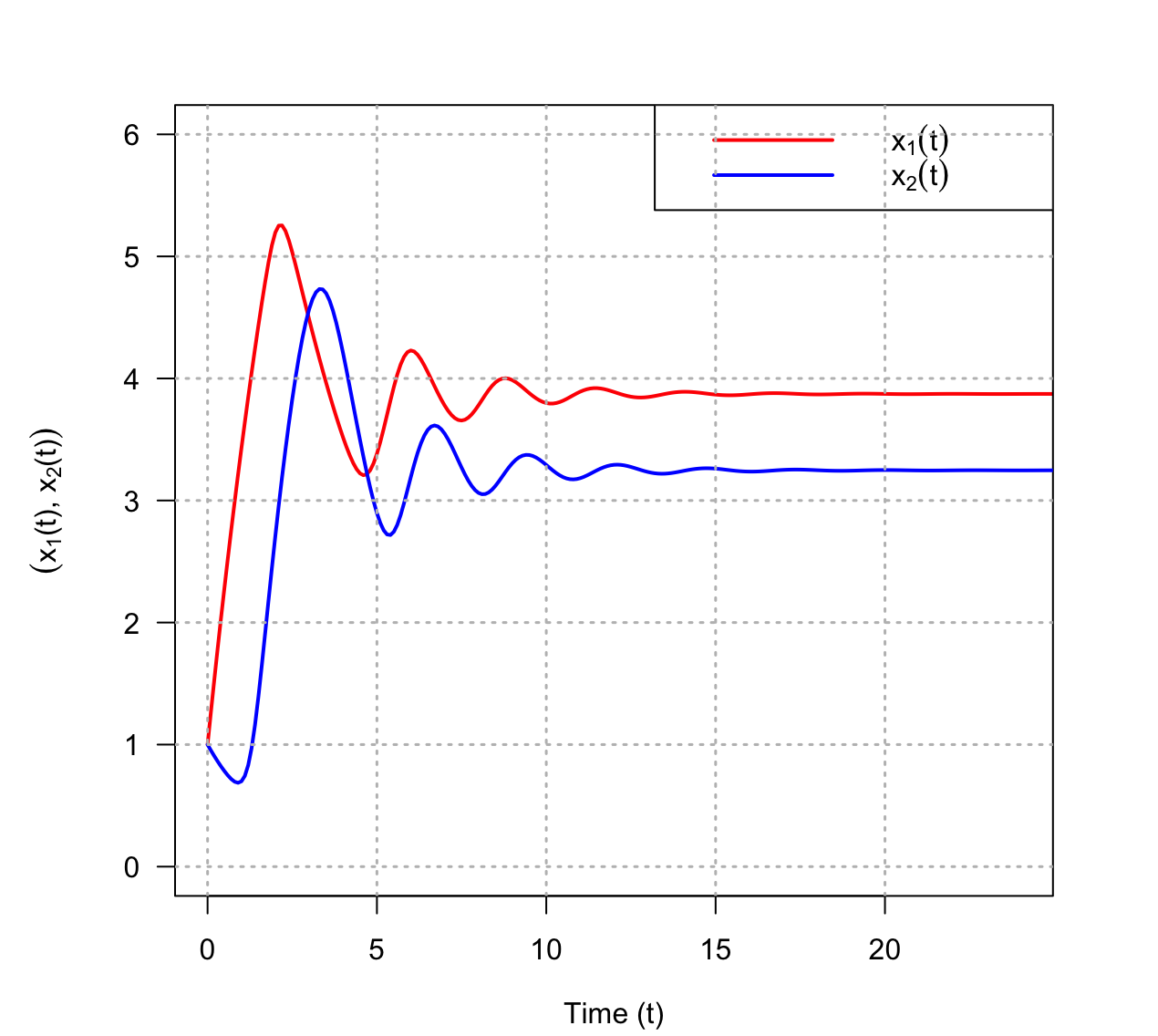}
    \caption{Temporal evolution of the two-gene oscillator system~\eqref{eq:oscillator}. Parameters: $\lambda = 3$, $\kappa_1 = 3$, $\gamma_1 = 0.25$, $\kappa_2 = 4$, $\gamma_2 = 0.5$, $\theta_1 = 4$, $\theta_2 = 3$; initial conditions $x_{01} = x_{02} = 1$.}
    \label{fig:Oscillateur_original}
\end{figure}

\subsubsection{Jacobian Matrix and Local Stability}

To analyse the stability properties of this system, we linearise the dynamics around an equilibrium point $(x_1^*, x_2^*)$ satisfying the steady-state conditions
\[
\kappa_1 f_1(x_2^*) = \gamma_1 x_1^*, \qquad \kappa_2 f_2(x_1^*) = \gamma_2 x_2^*,
\]
where we write $f_1(x_2) \equiv f^-(x_2,\theta_2,\lambda) = \frac{1}{1 + e^{\lambda(x_2 - \theta_2)}}$ (decreasing) and $f_2(x_1) \equiv f^+(x_1,\theta_1,\lambda) = \frac{1}{1 + e^{-\lambda(x_1 - \theta_1)}}$ (increasing) for compactness.

The Jacobian matrix at equilibrium is:
\[
J =
\begin{pmatrix}
-\gamma_1 & \kappa_1 f_1'(x_2^*) \\[4pt]
\kappa_2 f_2'(x_1^*) & -\gamma_2
\end{pmatrix}.
\]

A key advantage of the logistic formulation manifests immediately in the derivatives. For the repression function $f_1$, we have
\[
f_1'(x_2) = \frac{d}{dx_2}\left(\frac{1}{1 + e^{\lambda(x_2 - \theta_2)}}\right) = -\lambda f_1(x_2) \big(1 - f_1(x_2)\big) < 0,
\]
exploiting the characteristic self-referential form of the logistic derivative. The negative sign arises because $f_1$ decreases as $x_2$ increases. Similarly, for the activation function $f_2$:
\[
f_2'(x_1) = \lambda f_2(x_1) \big(1 - f_2(x_1)\big) > 0.
\]

These expressions are remarkably simple and require no fractional exponents or complex algebraic manipulations, a direct consequence of the logistic function's exponential structure.

\subsubsection{Eigenstructure}

The characteristic polynomial of $J$ is:
\[
\chi(\mu) = (\mu + \gamma_1)(\mu + \gamma_2) - \kappa_1 \kappa_2 f_1'(x_2^*) f_2'(x_1^*).
\]

Expanding this expression, we obtain the trace and determinant:
\begin{align*}
\operatorname{tr}(J) &= -(\gamma_1 + \gamma_2) < 0, \\[4pt]
\det(J) &= \gamma_1 \gamma_2 + \kappa_1 \kappa_2 \lambda^2 f_1(x_2^*)\big(1 - f_1(x_2^*)\big) f_2(x_1^*)\big(1 - f_2(x_1^*)\big) > 0.
\end{align*}

The trace is manifestly negative due to the positive degradation rates. The determinant is positive because all terms are non-negative: the product $f(1-f)$ attains its maximum value of $1/4$ at the inflection point but remains strictly positive throughout the interior of the unit interval.

\subsubsection{Local Asymptotic Stability}

The Routh--Hurwitz criterion for a $2 \times 2$ matrix guarantees local asymptotic stability when $\operatorname{tr}(J) < 0$ and $\det(J) > 0$, conditions that are satisfied here for all biologically meaningful parameter values. The nature of the approach to equilibrium depends on the discriminant:
\[
\Delta = \big(\operatorname{tr}(J)\big)^2 - 4\det(J).
\]

When $\Delta < 0$, the eigenvalues form a complex conjugate pair with negative real part $\operatorname{Re}(\mu) = -(\gamma_1 + \gamma_2)/2 < 0$, producing \emph{damped oscillations} where the system spirals into the equilibrium in a decaying sinusoidal manner.

\subsubsection{Bifurcation Landscape}

A natural question arises: Can this system undergo a Hopf bifurcation, transitioning from a stable equilibrium to sustained periodic oscillations as the parameters vary?

For a Hopf bifurcation to occur in a two-dimensional system, we require $\operatorname{tr}(J) = 0$ (the sum of eigenvalues vanishing so they become purely imaginary) while maintaining $\det(J) > 0$ (ensuring complex conjugates rather than real roots). However, in our model the trace
\[
\operatorname{tr}(J) = -(\gamma_1 + \gamma_2)
\]
is fixed and strictly negative for any positive degradation rates $\gamma_1, \gamma_2 > 0$. No variation of the remaining parameters $\lambda$, $\kappa_i$, or $\theta_i$ can force the trace to vanish, since these parameters enter only the off-diagonal entries of $J$.

The structural impossibility of a Hopf bifurcation, together with global stability, can be formalised as follows.

\begin{theorem}[Global asymptotic stability and absence of Hopf bifurcation]
\label{thm:no_hopf}
Let $\kappa_1,\kappa_2,\gamma_1,\gamma_2,\lambda,\theta_1,\theta_2>0$ and consider the two-gene logistic oscillator~\eqref{eq:oscillator} on $\mathbb{R}^2$.  Then:
\begin{enumerate}
\item[(i)] The system possesses a unique equilibrium $\mathbf{x}^{*}\in(0,\kappa_1/\gamma_1)\times(0,\kappa_2/\gamma_2)$.
\item[(ii)] $\mathbf{x}^{*}$ is locally asymptotically stable for \emph{every} choice of positive parameters; in particular, $\operatorname{tr}(J(\mathbf{x}^{*}))=-(\gamma_1+\gamma_2)<0$ and $\det(J(\mathbf{x}^{*}))>0$.
\item[(iii)] No choice of parameters $(\kappa_i,\gamma_i,\lambda,\theta_i)\in\mathbb{R}_{>0}^{6}$ admits a Hopf bifurcation: the system does not produce sustained limit cycles.
\item[(iv)] The closed box $\mathcal{B}_{\mathrm{osc}}=[0,\kappa_1/\gamma_1]\times[0,\kappa_2/\gamma_2]$ is forward invariant, and every trajectory of~\eqref{eq:oscillator} on $\mathbb{R}^2$ enters $\mathcal{B}_{\mathrm{osc}}$ in finite time.
\item[(v)] The divergence of the right-hand side of~\eqref{eq:oscillator} is identically $-(\gamma_1+\gamma_2)<0$ on $\mathbb{R}^2$; by Bendixson's negative criterion combined with Poincar\'e--Bendixson, the equilibrium $\mathbf{x}^{*}$ is \emph{globally} asymptotically stable on $\mathbb{R}^2$.
\end{enumerate}
\end{theorem}

\begin{proof}
\emph{(i)} An equilibrium $\mathbf{x}^{*}=(x_1^{*},x_2^{*})$ satisfies
$x_1=(\kappa_1/\gamma_1)f^{-}(x_2)$ and $x_2=(\kappa_2/\gamma_2)f^{+}(x_1)$.
Composing the two relations defines the map\\
$T\colon[0,\kappa_1/\gamma_1]\to[0,\kappa_1/\gamma_1]$,
$T(x_1) = (\kappa_1/\gamma_1)\,f^{-}\!\bigl((\kappa_2/\gamma_2)\,f^{+}(x_1)\bigr)$.
Since $f^{+\prime}>0$ and $f^{-\prime}<0$, the composition is strictly
decreasing, so $T(x_1)-x_1$ is strictly decreasing on
$[0,\kappa_1/\gamma_1]$.  At $x_1=0$,
$T(0) = (\kappa_1/\gamma_1)\,f^{-}\bigl((\kappa_2/\gamma_2)f^{+}(0)\bigr) > 0$
because $f^{-}>0$ everywhere; at $x_1=\kappa_1/\gamma_1$,
$T(\kappa_1/\gamma_1) < \kappa_1/\gamma_1$ because $f^{-}<1$.  The
intermediate-value theorem then gives a unique fixed point
$x_1^{*}\in(0,\kappa_1/\gamma_1)$, and the corresponding
$x_2^{*}=(\kappa_2/\gamma_2)f^{+}(x_1^{*})$ lies in
$(0,\kappa_2/\gamma_2)$ by an analogous argument.

\emph{(ii)} The Jacobian at $\mathbf{x}^{*}$ has off-diagonal entries
$J_{12}=\kappa_1 f_1'(x_2^{*})<0$ and $J_{21}=\kappa_2 f_2'(x_1^{*})>0$.
Hence $\det J = \gamma_1\gamma_2 - J_{12}J_{21}>\gamma_1\gamma_2>0$,
since $J_{12}J_{21}<0$.  Combined with
$\operatorname{tr}J=-(\gamma_1+\gamma_2)<0$, the Routh--Hurwitz criterion
gives local asymptotic stability.

\emph{(iii)} A Hopf bifurcation in $\mathbb{R}^{2}$ requires a smooth
one-parameter family of equilibria along which
$\operatorname{tr}J$ changes sign while $\det J$ remains
positive~\citep{Marsden1976,Kuznetsov2004}.  Since
$\operatorname{tr}J=-(\gamma_1+\gamma_2)$ depends only on
$(\gamma_1,\gamma_2)\in\mathbb{R}_{>0}^{2}$ and is strictly negative
throughout this open quadrant, no such family exists.

\emph{(iv)} On the boundary face $\{x_1=0\}$ of $\mathcal{B}_{\mathrm{osc}}$,
$\dot x_1 = \kappa_1 f^{-}(x_2) > 0$; on $\{x_1=\kappa_1/\gamma_1\}$,
$\dot x_1 = \kappa_1 f^{-}(x_2) - \gamma_1(\kappa_1/\gamma_1) =
\kappa_1(f^{-}(x_2)-1) \le 0$.  The same inequalities hold for $x_2$, so
$\mathcal{B}_{\mathrm{osc}}$ is forward invariant by Nagumo's
theorem~\citep{blanchini2008set}.  Outside $\mathcal{B}_{\mathrm{osc}}$,
since $f^{\pm}\in(0,1)$ we have $\dot x_i\le\kappa_i-\gamma_i x_i$ (strictly
negative when $x_i>\kappa_i/\gamma_i$) and $\dot x_i\ge-\gamma_i x_i$
(strictly positive when $x_i<0$), so any trajectory enters
$\mathcal{B}_{\mathrm{osc}}$ in finite time.

\emph{(v)} The divergence is
$\partial \dot x_1/\partial x_1 + \partial \dot x_2/\partial x_2 =
-\gamma_1 - \gamma_2 < 0$ identically on $\mathbb{R}^2$, so by Bendixson's
negative criterion no closed orbit (periodic orbit or homoclinic loop)
can lie in any simply connected region of $\mathbb{R}^2$.  By the
Poincar\'e--Bendixson theorem applied to the forward-invariant compact
region $\mathcal{B}_{\mathrm{osc}}$, the $\omega$-limit set of every
trajectory is a single equilibrium; uniqueness from~(i) and local
asymptotic stability from~(ii) force this equilibrium to be
$\mathbf{x}^{*}$.  Combined with~(iv), every trajectory of~\eqref{eq:oscillator} on $\mathbb{R}^2$
converges to $\mathbf{x}^{*}$, so $\mathbf{x}^{*}$ is globally
asymptotically stable.
\end{proof}

\begin{remark}
Theorem~\ref{thm:no_hopf} shows that the \emph{two-gene} negative-feedback motif modelled by~\eqref{eq:oscillator} cannot, by itself, generate sustained oscillations: the obstruction is the planar Bendixson criterion and is therefore specific to dimension two.  Two routes restore sustained oscillation.  Enlarging the network beyond two genes lifts the planar obstruction---the eleven-gene Traynard system of Section~\ref{ex:traynard} exhibits a delay-free limit cycle; alternatively, while retaining only two genes, sustained limit cycles emerge once delay differential equations are incorporated to account for transcription, translation, or transport lags~\citep{belgacem2025glass,farcot2019chaos,belgacem2020control,belgacem2026sustained,belgacem2026beyond}, as in real biological clocks.  We develop the delayed two-gene logistic oscillator in companion work~\citep{belgacem2026sustained,belgacem2026beyond}.
\end{remark}

\subsection{A Bistable Motif: The Genetic Toggle Switch}
\label{ex:toggle}

The two-gene oscillator of Section~\ref{ex:oscillator} pairs one activation with
one repression. Replacing the activation by a second repression yields the other
fundamental two-gene motif, mutual repression. Its canonical realisation is the
\emph{genetic toggle switch} of~\citet{gardner2000construction}, a synthetic
bistable network built in \textit{Escherichia coli} from two repressible
promoters, each gene's product repressing the other; the same mutually
inhibitory architecture underlies the natural bacteriophage~$\lambda$
lysis--lysogeny switch. The toggle is the archetype of cellular memory: it
possesses two stable expression states and retains whichever one it was last
driven into.

The classical model of~\citet{gardner2000construction} represents repression by
Hill functions. Translating the same Boolean wiring---$b_1=\lnot b_2$ and
$b_2=\lnot b_1$---through the logistic framework replaces each Hill term by a
decreasing logistic function and gives
\begin{equation}
\label{eq:toggle}
\dot{x}_1 = \kappa_1\,f^-(x_2,\theta_2,\lambda) - \gamma_1 x_1,
\qquad
\dot{x}_2 = \kappa_2\,f^-(x_1,\theta_1,\lambda) - \gamma_2 x_2 ,
\end{equation}
with $\kappa_i,\gamma_i,\theta_i,\lambda>0$. The following proposition gives a
complete and rigorous account of its dynamics.

\begin{proposition}[Dynamics of the logistic toggle switch]
\label{prop:toggle}
For system~\eqref{eq:toggle}, write $\rho_i=\kappa_i/\gamma_i$. Then:
\begin{enumerate}
\item[(i)] the box $\mathcal{B}=[0,\rho_1]\times[0,\rho_2]$ is forward invariant
  and globally attracting, and the system has no periodic orbit;
\item[(ii)] at every equilibrium the Jacobian satisfies
  $\operatorname{tr}J=-(\gamma_1+\gamma_2)<0$ and has two real eigenvalues; the
  equilibrium is a stable node when $\det J>0$ and a saddle when $\det J<0$, so
  no equilibrium is a focus and the switch cannot oscillate;
\item[(iii)] the equilibria are the fixed points of the strictly increasing map
  $T(x_1)=\rho_1 f^-\!\bigl(\rho_2 f^-(x_1,\theta_1,\lambda),\theta_2,\lambda\bigr)$,
  and at each one $\det J=\gamma_1\gamma_2\,\bigl(1-T'(x_1^{*})\bigr)$;
  consequently the equilibrium is a stable node iff $T'(x_1^{*})<1$ and a saddle
  iff $T'(x_1^{*})>1$. Hyperbolic equilibria alternate node--saddle--node along
  the graph of $T$; generically the system is either \emph{monostable} (one
  equilibrium, globally asymptotically stable) or \emph{bistable} (two stable
  nodes separated by one saddle);
\item[(iv)] in the symmetric case $\kappa_1=\kappa_2$, $\gamma_1=\gamma_2$,
  $\theta_1=\theta_2$ (write $\rho=\kappa/\gamma$) there is exactly one symmetric
  equilibrium $x_s$, and the system is bistable iff
  $\rho\lambda\,f^-(x_s)\,\bigl(1-f^-(x_s)\bigr)>1$. If moreover $\theta=\rho/2$
  then $x_s=\rho/2$ and the system undergoes a supercritical pitchfork
  bifurcation at $\rho\lambda=4$: it is monostable for $\rho\lambda<4$ and
  bistable for $\rho\lambda>4$.
\end{enumerate}
\end{proposition}

\begin{proof}
\emph{(i)} On the face $\{x_i=0\}$ one has $\dot{x}_i=\kappa_i f^->0$, and on
$\{x_i=\rho_i\}$ one has $\dot{x}_i=\kappa_i(f^--1)\le0$; by Nagumo's theorem
$\mathcal{B}$ is forward invariant, and since $\dot{x}_i<0$ for $x_i>\rho_i$ and
$\dot{x}_i>0$ for $x_i<0$, every trajectory enters $\mathcal{B}$. The divergence
of the vector field is $-(\gamma_1+\gamma_2)<0$ throughout $\mathbb{R}^2$, so by
Bendixson's criterion no periodic orbit exists. (Forward invariance is also the
two-gene instance of Proposition~\ref{prop:wellposedness}.)

\emph{(ii)} Since $\partial_x f^-(x,\theta,\lambda)=-\lambda f^-(1-f^-)$, the
Jacobian of~\eqref{eq:toggle} is
\[
  J=\begin{pmatrix}-\gamma_1 & -\kappa_1 g_2\\[2pt] -\kappa_2 g_1 & -\gamma_2\end{pmatrix},
  \qquad
  g_i=\lambda\,f^-(x_i,\theta_i,\lambda)\bigl(1-f^-(x_i,\theta_i,\lambda)\bigr)>0,
\]
so $\operatorname{tr}J=-(\gamma_1+\gamma_2)$ and
$\det J=\gamma_1\gamma_2-\kappa_1\kappa_2 g_1 g_2$. The discriminant of the
characteristic polynomial,
$(\operatorname{tr}J)^2-4\det J=(\gamma_1-\gamma_2)^2+4\kappa_1\kappa_2 g_1 g_2$,
is strictly positive, so both eigenvalues are real; with $\operatorname{tr}J<0$
they are both negative exactly when $\det J>0$ (stable node) and of opposite
sign when $\det J<0$ (saddle).

\emph{(iii)} An equilibrium satisfies $x_1=\rho_1 f^-(x_2,\theta_2,\lambda)$ and
$x_2=\rho_2 f^-(x_1,\theta_1,\lambda)$; eliminating $x_2$ gives $x_1=T(x_1)$. As
the composition of two strictly decreasing logistic maps with positive
rescalings, $T$ is strictly increasing and real-analytic, with $T(0)>0$ and
$T(\rho_1)<\rho_1$; hence $T(x_1)-x_1$ has an odd number of zeros, all isolated.
Differentiating, $T'(x_1^{*})=\rho_1\rho_2\,g_1 g_2$, so
$\det J=\gamma_1\gamma_2-\kappa_1\kappa_2 g_1 g_2=\gamma_1\gamma_2\bigl(1-T'(x_1^{*})\bigr)$,
and the dichotomy of~(ii) becomes $T'(x_1^{*})\lessgtr1$. Because $T$ is
increasing with $T-x_1$ positive at $0$ and negative at $\rho_1$, consecutive
zeros are crossings of alternating direction: the outermost are down-crossings
($T'<1$, stable nodes) and any interior zero is an up-crossing ($T'>1$, saddle).
A single zero is therefore a stable node; by~(i) the box $\mathcal{B}$ is forward
invariant and free of periodic orbits, so the Poincar\'e--Bendixson theorem
forces every trajectory to converge to it, giving global asymptotic stability.
Three zeros give the node--saddle--node configuration of a bistable switch.

\emph{(iv)} In the symmetric case the exchange $x_1\leftrightarrow x_2$ leaves
the system invariant; a symmetric equilibrium solves
$x_s=\rho\,f^-(x_s,\theta,\lambda)$, whose right-hand side is strictly decreasing
in $x_s$, so it is unique. There $g_1=g_2=\lambda f^-(x_s)(1-f^-(x_s))$ and
$T'(x_s)=\bigl[\rho\lambda f^-(x_s)(1-f^-(x_s))\bigr]^2$, so by~(iii) the
symmetric equilibrium is a saddle---and the switch bistable---iff
$\rho\lambda f^-(x_s)(1-f^-(x_s))>1$. If $\theta=\rho/2$, then
$f^-(\rho/2,\rho/2,\lambda)=\tfrac12$ gives $x_s=\rho/2$ and
$f^-(x_s)(1-f^-(x_s))=\tfrac14$, so the condition reduces to $\rho\lambda>4$. The
$x_1\leftrightarrow x_2$ symmetry makes the loss of stability of the symmetric
branch at $\rho\lambda=4$ a pitchfork bifurcation; the two emerging branches are
stable nodes present only for $\rho\lambda>4$, so it is supercritical.
\end{proof}

The identity $\det J=\gamma_1\gamma_2\bigl(1-T'(x_1^{*})\bigr)$ in part~(iii) is
the quantitative core of the result: it ties the \emph{dynamical} stability of
an equilibrium of the two-dimensional flow to the \emph{geometric} slope with
which the scalar map $T$ crosses the diagonal, and shows that bistability is
created exactly when that slope first exceeds unity. Part~(iv) makes the
governing parameter explicit---the product $\rho\lambda$ of the response
steepness and the production-to-degradation ratio---and identifies
$\rho\lambda=4$ as the sharp monostable/bistable threshold, a closed-form
criterion that the non-smooth Hill formulation does not furnish as cleanly.

The toggle also provides the smallest non-trivial illustration of the Boolean
fixed-point recovery theorem. The Boolean toggle $b_1=\lnot b_2$, $b_2=\lnot b_1$
has exactly two fixed points, $(1,0)$ and $(0,1)$; the states $(0,0)$ and
$(1,1)$ are not fixed. Whenever $0<\theta_i<\rho_i$,
Theorem~\ref{thm:boolean_recovery} applies and predicts that, for sufficiently
steep response, each of these two Boolean steady states is realised as an
exponentially stable equilibrium of~\eqref{eq:toggle}, with coordinates
approaching $(\rho_1,0)$ and $(0,\rho_2)$. The bistability established in
Proposition~\ref{prop:toggle} is exactly this two-state recovery, and part~(iv)
additionally locates the response steepness at which it switches on.

Figure~\ref{fig:toggle} confirms the analysis by direct integration
of~\eqref{eq:toggle} with a fourth-order Runge--Kutta scheme implemented in R,
for the symmetric parameters $\kappa=10$, $\gamma=1$, $\theta=5$ (so $\rho=10$,
$\theta=\rho/2$, and $\lambda^{\star}=4/\rho=0.4$). Panel~(a) shows the bistable
phase portrait at $\lambda=2$: the two nullclines intersect in three points, and
trajectories launched from a grid of initial conditions partition into two
basins, each converging to one of the stable nodes. The equilibria located
numerically and classified by their Jacobian spectra agree with
Proposition~\ref{prop:toggle}---the stable nodes sit at $(9.9996,\,0.0005)$ and
its mirror image, with $\det J=+1.00$ and real eigenvalues $\{-1.001,-0.999\}$,
while the central saddle at $(5,5)$ has $\det J=-24$ and real eigenvalues
$\{4,-6\}$; no eigenvalue has a nonzero imaginary part, confirming the absence
of oscillation. Each stable node lies within $5\times10^{-4}$ (in the
$\ell^{\infty}$ norm) of its respective Boolean vertex, $(\rho,0)$ and
$(0,\rho)$, well within the bound $e^{-\lambda\delta/2}=6.7\times10^{-3}$ of
Theorem~\ref{thm:boolean_recovery} (here $\delta=\rho/2=5$). Panel~(b) traces
the equilibria as $\lambda$ varies: a single stable branch for $\lambda<0.4$
splits, at $\lambda^{\star}=0.4$, into a saddle and two stable branches---the
supercritical pitchfork of part~(iv)---confirming $\rho\lambda=4$ as the exact
threshold.

\begin{figure}[htbp]
\centering
\includegraphics[width=\linewidth]{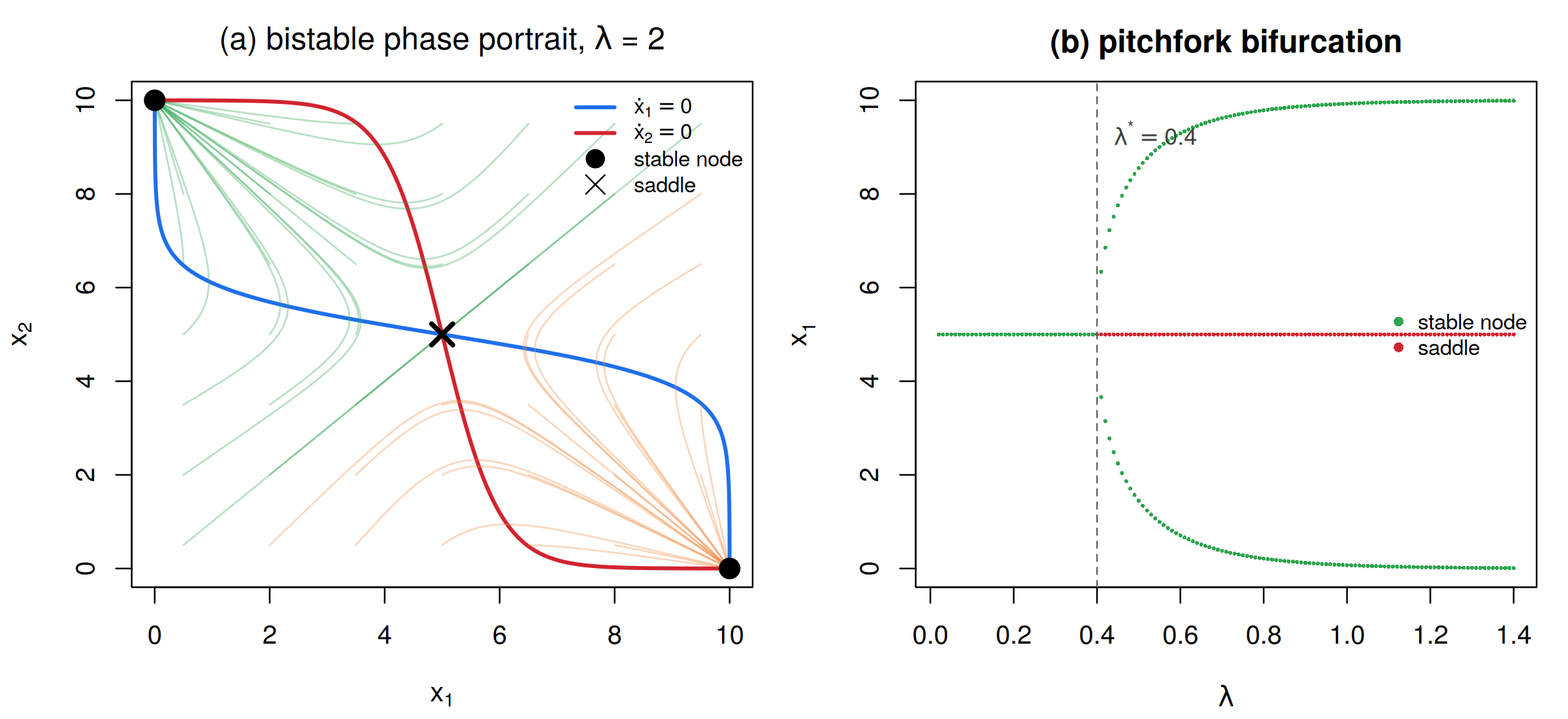}
\caption{Logistic genetic toggle switch~\eqref{eq:toggle} integrated in R, with
symmetric parameters $\kappa=10$, $\gamma=1$, $\theta=5$. \textbf{(a)}~Bistable
phase portrait at $\lambda=2$: nullclines $\dot{x}_1=0$ (blue) and
$\dot{x}_2=0$ (red), the two stable nodes (filled circles), the saddle (cross),
and trajectories from a grid of initial conditions, coloured by the basin they
reach. \textbf{(b)}~Bifurcation diagram: the equilibrium coordinate $x_1^{*}$
against the steepness $\lambda$, showing the supercritical pitchfork at
$\lambda^{\star}=4/\rho=0.4$ (stable nodes in green, saddle in red).}
\label{fig:toggle}
\end{figure}

\subsection{Reduction to the Standard Logistic Form}
\label{subsec:standard_form}

Every logistic-based regulatory function can be recast in terms of the parameter-free standard logistic $f(s) = 1/(1+e^{-s})$ through an affine change of variable. For the genetic oscillator~\eqref{eq:oscillator}, introducing the affine arguments $z_1 = \lambda(\theta_2 - x_2)$ and $z_2 = \lambda(x_1 - \theta_1)$ rewrites the dynamics in the canonical form $\dot{x}_i = \kappa_i f(z_i) - \gamma_i x_i$ for $i = 1, 2$. The same reduction extends to arbitrary networks: every sigmoidal contribution becomes a value of the same standard function $f$, with all biological information carried by the affine maps $z = \lambda(x - \theta)$ or $z = \lambda(\theta - x)$.

This canonical form exposes a useful symmetry. The standard logistic satisfies $f(-s) = 1 - f(s)$, so the decreasing sigmoid encoding repression is the complement of the increasing one encoding activation, and the two regulatory modes admit a unified probabilistic interpretation: $f(s)$ is the activation probability and $1-f(s)$ the repression probability. The reduction simplifies stability analysis, linearisation, and bifurcation calculations, and makes opposing regulatory influences manifestly consistent within a single mathematical framework.

\section{Multi-Gene Regulatory Networks}
\label{subsec:multi-gene}

The logistic-based formulation extends naturally to arbitrarily large networks,
preserving its analytical and computational advantages even in high-dimensional
settings.  A distinguishing feature of our modelling philosophy is the
\emph{explicit differentiation} between increasing and decreasing logistic
functions: we deploy each form precisely where it is biologically appropriate,
thereby maintaining clear regulatory interpretability.  This stands in contrast
to some alternative formulations that attempt to unify all regulatory
interactions through a single increasing functional
form~\citep{samuilik2022mathematical}, potentially obscuring biological meaning.

Our framework accommodates \emph{logical rules} governing how multiple
regulators collectively orchestrate gene expression, following the Boolean
network formalism pioneered by \citet{albert2003topology} for
the \textit{Drosophila} segment polarity network.  These rules capture
combinatorial effects that arise naturally in transcriptional regulation.
We denote by $\Phi\colon \mathbb{R}^N \to [0,1]$ the map that translates
a Boolean regulatory rule into its continuous logistic approximation.

\subsection{General Multi-Gene Formulation}

Consider a gene regulatory network comprising $N$ genes, where $x_i(t)$
denotes the concentration of the protein product of gene $i$ at time $t$.
Each gene's dynamics obey a balance equation between synthesis and
degradation.  Each variable $x_i$ then evolves according to
\begin{equation}
    \dot{x}_i \;=\; \kappa_i\,\Phi_i(\mathbf{x}) \;-\; \gamma_i\,x_i,
    \label{eq:ode_0}
\end{equation}
where $\Phi_i\colon \mathbb{R}^{N} \to [0,1]$
is the regulatory function that synthesises the influences from all relevant
activators and repressors in the network: it is the continuous logistic
approximation of the Boolean rule $\varphi_i$ (range and structure
established in Proposition~\ref{prop:demorgan} below), $\kappa_i > 0$ is
the maximal production rate, and $\gamma_i > 0$ is the degradation rate.
Every trajectory that starts in the box
$\mathcal{B}=\prod_i[0,\kappa_i/\gamma_i]$ remains in it (forward invariance,
Proposition~\ref{prop:wellposedness} below), so that
$x_i(t) \leq \kappa_i/\gamma_i$ for all $t \geq 0$; the box is moreover globally
attracting, so a trajectory launched outside it enters in finite time.  This
bound is uniform in the network size.  The map
$\Phi\colon \mathbb{R}^N \to [0,1]$ is defined recursively on the structure
of the rule as follows.

\begin{itemize}

  \item \textbf{Constant \textsc{False}} (gene permanently silenced):
        \[
            \Phi(\textsc{False}) \;\equiv\; 0,
        \]
        so that~\eqref{eq:ode_0} reduces to $\dot{x}_i = -\gamma_i x_i$,
        i.e.\ pure exponential decay to zero.

  \item \textbf{Constant \textsc{True}} (constitutive expression):
        \[
            \Phi(\textsc{True}) \;\equiv\; 1,
        \]
        so that~\eqref{eq:ode_0} reduces to $\dot{x}_i = \kappa_i - \gamma_i x_i$, which converges monotonically to the constitutive steady state $x_i^* = \kappa_i/\gamma_i$.

  \item \textbf{Positive literal} (activating signal): a variable $x_j$
        is mapped to the increasing logistic function,
        \[
            \Phi(x_j) \;=\; f^+(x_j).
        \]

  \item \textbf{NOT logic} (repression): a negative literal $\lnot x_j$
        is mapped to the decreasing logistic function,
        \begin{equation}
            \Phi(\lnot x_j)
            \;=\;
            f^-(x_j,\theta_j,\lambda)
            \;=\;
            1 - f^+(x_j),
            \label{eq:not_logic}
        \end{equation}
        consistent with the complement structure of Boolean negation.

  \item \textbf{AND logic} (cooperative co-regulation): a conjunction
        $C(\mathbf{x}) = x_{i_1} \wedge \cdots \wedge x_{i_k}$ is mapped
        to the \emph{product} of the corresponding logistic terms,
        \[
            \Phi\!\left(C(\mathbf{x})\right)
            \;=\;
            \Phi(x_{i_1} \wedge \cdots \wedge x_{i_k})
            \;=\;
            \prod_{l=1}^{k} f^+(x_{i_l}),
        \]
        modelling the requirement for the simultaneous satisfaction of all
        regulatory conditions.

  \item \textbf{OR logic} (independent activation): in the elementary case
        of two activating literals the map reads
        \begin{equation}
            \Phi(x_1 \vee x_2)
            \;=\;
            1 - \bigl(1 - f^+(x_1)\bigr)\bigl(1 - f^+(x_2)\bigr)
            \;=\;
            1 - f^-(x_1)\,f^-(x_2),
            \label{eq:or_elementary}
        \end{equation}
        which generalises to a disjunction of arbitrary clauses
        $C_1(\mathbf{x}) \vee C_2(\mathbf{x})$ as
        \[
            \Phi\!\left(C_1(\mathbf{x}) \vee C_2(\mathbf{x})\right)
            \;=\;
            1 - \bigl(1 - \Phi(C_1(\mathbf{x}))\bigr)
                \bigl(1 - \Phi(C_2(\mathbf{x}))\bigr),
        \]
        ensuring $\Phi \in [0,1]$ regardless of the number of independent
        regulatory pathways, in contrast to additive formulations that
        violate the unit bound when multiple activators are simultaneously
        active.  In the general case, a disjunction
        $C_1(\mathbf{x}) \vee \cdots \vee C_m(\mathbf{x})$ is mapped via
        the recursive De~Morgan product formula
        \begin{equation}
            \Phi\!\left(\bigvee_{k=1}^{m} C_k(\mathbf{x})\right)
            \;=\;
            1 - \prod_{k=1}^{m}\bigl(1 - \Phi(C_k(\mathbf{x}))\bigr),
            \label{eq:demorgan_0}
        \end{equation}
        which follows from the classical De~Morgan law
        $\lnot\!\left(\lnot C_1(\mathbf{x}) \wedge \cdots \wedge
        \lnot C_m(\mathbf{x})\right)$ applied recursively, and coincides
        with the probability that at least one of $m$ independent events
        occurs~\citep{feller1971introduction}.
        \citet{wittmann2009transforming} implicitly
        recover the two-input instance
        $\Phi(x_1 \vee x_2) = x_1 + x_2 - x_1 x_2$
        via multivariate polynomial interpolation of the Boolean OR gate,
        but do not state the general $m$-clause recursive product
        form~\eqref{eq:demorgan_0}.
        The bound-violation risk arises in naive additive translations of
        Boolean OR, where
        $\Phi_i = \sum_{k=1}^{m}\Phi(C_k(\mathbf{x}))$ can reach $m$ when
        all clauses are simultaneously active, inflating production rates to
        $m \cdot \kappa_i$ and destroying the biological bound
        $x_i^* \leq \kappa_i/\gamma_i$.

  \item \textbf{Hybrid combinations}: more complex regulatory
        architectures involving mixtures of AND, OR, and NOT gates,
        as encountered in developmental gene
        networks~\citep{albert2003topology} and synthetic biology
        circuits~\citep{nielsen2016genetic}.  For hybrid combinations,
        well-established Boolean models exist, including breast
        cancer~\citep{biane2018causal}, cell
        cycle~\citep{Traynard2016}, ERBB-regulated G1/S
        transition~\citep{Sahin2009}, S-phase entry and
        senescence~\citep{Verlingue2016}, epithelial-to-mesenchymal
        transition~\citep{Cohen2015}, and models of the pro-inflammatory
        tumour microenvironment in acute lymphoblastic
        leukaemia~\citep{Enciso2016}, among others.

\end{itemize}

The recursive map $\Phi$ defined above behaves as a unit-interval
\emph{soft} extension of the Boolean operations.  The following
proposition records its three essential structural properties: it sends
Boolean values to the interval endpoints, it is consistent with the
classical Boolean operations, and it produces strictly interior values
for any finite expression and any real-valued state.  The third property
is what guarantees strict basal expression and is therefore central to
the absence of the trapping pathology, which is analysed in detail in the companion paper~\citep{belgacem2026numerical}.

\begin{proposition}[Structural properties of the De~Morgan map $\Phi$]
\label{prop:demorgan}
Let $\Phi$ be defined recursively on Boolean formulae by the rules
$\Phi(\mathrm{True})\equiv 1$, $\Phi(\mathrm{False})\equiv 0$,
$\Phi(x_j)=f^{+}(x_j,\theta_j,\lambda)$,
$\Phi(\lnot\varphi)=1-\Phi(\varphi)$,
$\Phi(\varphi\wedge\psi)=\Phi(\varphi)\,\Phi(\psi)$, and
$\Phi(\varphi\vee\psi)=1-\bigl(1-\Phi(\varphi)\bigr)\bigl(1-\Phi(\psi)\bigr)$,
with $\lambda,\theta_j>0$.  Then for any Boolean formula $\varphi$
involving finitely many variables $x_1,\ldots,x_N$ and any
$\mathbf{x}\in\mathbb{R}^N$:
\begin{enumerate}
\item[(i)] (\emph{Range.})  $\Phi(\varphi)(\mathbf{x})\in[0,1]$, and
  $\Phi(\varphi)(\mathbf{x})\in(0,1)$ \emph{strictly} whenever $\varphi$
  is neither identically $\mathrm{True}$ nor identically $\mathrm{False}$.
\item[(ii)] (\emph{Boolean consistency, with uniform convergence.})  If each
  $x_j\in\{0,1\}$ is identified with the limiting Boolean value
  $\lim_{\lambda\to\infty}f^{+}(x_j,\theta_j,\lambda)
  =\mathbf{1}_{\{x_j>\theta_j\}}\in\{0,1\}$, then in this limit $\Phi$
  reduces to the classical Boolean evaluation map.  Moreover, on any
  compact set $K\subset\mathbb{R}^{N}$ bounded away from the threshold
  hyperplanes $\{x_j=\theta_j\}$ by a distance at least $\delta>0$, the
  convergence is uniform: there exists $C_{\varphi}>0$, depending only
  on the structure of $\varphi$, such that
  \[
    \sup_{\mathbf{x}\in K}\bigl|\Phi(\varphi)(\mathbf{x})-
    \mathbf{1}_{\{\varphi\text{ holds at }\mathbf{x}_{\mathrm{disc}}\}}\bigr|
    \;\leq\; C_{\varphi}\,e^{-\lambda\delta}\;\to\;0
    \quad\text{as }\lambda\to\infty,
  \]
  where $\mathbf{x}_{\mathrm{disc}}=(\mathbf{1}_{\{x_j>\theta_j\}})_{j=1}^{N}$.
\item[(iii)] (\emph{De~Morgan duality.})  Replacing every literal $x_j$
  by its negation $\lnot x_j$ swaps the roles of $\wedge$ and $\vee$
  under $\Phi$:
  \[
    \Phi\bigl(\lnot(\varphi\wedge\psi)\bigr)
    = \Phi(\lnot\varphi\vee\lnot\psi),
    \qquad
    \Phi\bigl(\lnot(\varphi\vee\psi)\bigr)
    = \Phi(\lnot\varphi\wedge\lnot\psi).
  \]
\end{enumerate}
In particular, the multi-clause OR formula~\eqref{eq:demorgan_0}
follows by recursive application of $(\wedge,\vee,\lnot)$ from the
elementary rules.
\end{proposition}

\begin{proof}
\emph{(i)} Each $f^{+}(x_j,\theta_j,\lambda)\in(0,1)$ for every
$x_j\in\mathbb{R}$ since the exponential is strictly positive and
finite.  The product, complement, and De~Morgan-product operations all
preserve $[0,1]$, and preserve the open subinterval $(0,1)$ whenever no
factor is identically $0$ or $1$.  The constants $\mathrm{True}$ and
$\mathrm{False}$ are the only formulae mapped to the endpoints.

\emph{(ii)} The pointwise limit
$\lim_{\lambda\to\infty}f^{+}(x_j,\theta_j,\lambda)$ is the Heaviside
indicator $\mathbf{1}_{\{x_j>\theta_j\}}$ (with value $1/2$ at
$x_j=\theta_j$, a measure-zero coincidence).  Products of indicators
are themselves indicators of intersections (corresponding to AND), and
$1-\prod(1-\mathbf{1}_{A_k})=\mathbf{1}_{\bigcup A_k}$ (corresponding to
OR), so the limit of $\Phi$ is the classical Boolean evaluation.
For the uniform statement, on $\{x_j\geq\theta_j+\delta\}$ we have
$|f^{+}(x_j,\theta_j,\lambda)-1|=1/(1+e^{\lambda(x_j-\theta_j)})\leq e^{-\lambda\delta}$,
and on $\{x_j\leq\theta_j-\delta\}$ we have
$f^{+}(x_j,\theta_j,\lambda)=1/(1+e^{\lambda(\theta_j-x_j)})\leq e^{-\lambda\delta}$.
Hence each factor is uniformly within $e^{-\lambda\delta}$ of its
Heaviside limit on $K$.  Since $\Phi(\varphi)$ is built from at most
finitely many factors via products and complements (each Lipschitz with
constant $1$ on $[0,1]$), the uniform error is bounded by
$C_{\varphi}\,e^{-\lambda\delta}$ where $C_{\varphi}$ counts the number
of literals in the formula.

\emph{(iii)} Set $p=\Phi(\varphi)$, $q=\Phi(\psi)$.  Then
\[
\Phi\bigl(\lnot(\varphi\wedge\psi)\bigr) = 1-\Phi(\varphi\wedge\psi) = 1-pq,
\]
while, by the OR rule applied to $\lnot\varphi$ and $\lnot\psi$,
\[
\Phi(\lnot\varphi\vee\lnot\psi) = 1-\bigl(1-\Phi(\lnot\varphi)\bigr)\bigl(1-\Phi(\lnot\psi)\bigr) = 1-p\,q,
\]
so the two sides agree.  The dual identity
$\Phi(\lnot(\varphi\vee\psi))=\Phi(\lnot\varphi\wedge\lnot\psi)$
follows symmetrically: both equal $(1-p)(1-q)$.
\end{proof}

Property~(i) of Proposition~\ref{prop:demorgan} is the analytical
foundation of basal expression: since $\Phi_i\in(0,1)$ \emph{strictly}
for any finite-$\lambda$ logistic system, the production term
$\kappa_i\Phi_i$ is bounded away from zero throughout the state space,
and the boundary face $\{x_i=0\}$ admits the strict inflow
$\dot{x}_i=\kappa_i\Phi_i>0$ used in the proof of forward invariance
(Proposition~\ref{prop:wellposedness}).  Property~(ii) guarantees that
the logistic continuous extension converges to the Boolean network in
the appropriate steepness limit; this convergence is the analytical basis of
the Boolean fixed-point recovery established in
Theorem~\ref{thm:boolean_recovery} below, which shows that every steady state
of the Boolean network reappears as an exponentially stable equilibrium of the
continuous system.  Property~(iii) is the precise structural reason why translating a
disjunction by the recursive De~Morgan product
formula~\eqref{eq:demorgan_0} is consistent with the negation rule:
the soft AND and soft OR are conjugate under the involution
$p\mapsto 1-p$.  We caution, however, that the soft map is \emph{not}
invariant under distributivity or idempotence: for instance,
$\Phi(\varphi\wedge(\psi\vee\chi))$ and
$\Phi((\varphi\wedge\psi)\vee(\varphi\wedge\chi))$ differ in general
even though the underlying Boolean formulas are equivalent (a direct
computation gives
$p(q+r-qr)$ versus $pq+pr-p^{2}qr$, with strict inequality whenever
$p,q,r\in(0,1)$).  This is the well-known caveat of any continuous
extension of Boolean logic to $[0,1]$~\citep{wittmann2009transforming}
and is the structural justification for the canonical DNF preprocessing
performed by \texttt{BooleanMinimize} in the large-scale numerical experiments of the companion paper~\citep{belgacem2026numerical}:
each rule is reduced to a unique minimal DNF representative before
translation, ensuring that the continuous $\Phi$ is well-defined as a
function of the Boolean rule rather than its syntactic form.

As an illustrative and biologically prevalent case, we examine
\textbf{parallel regulation}, in which multiple transcription factors
simultaneously exert independent regulatory effects, some activating and
others repressing the target gene's promoter.  In this architecture the
regulatory function takes a product structure:
\begin{equation}
    f_i(x_1, \ldots, x_N)
    \;=\;
    \prod_{j \in \mathcal{A}_i}
        \frac{1}{1 + e^{-\lambda(x_j - \theta_{ij})}}
    \;\cdot\;
    \prod_{k \in \mathcal{R}_i}
        \frac{1}{1 + e^{-\lambda(\theta_{ik} - x_k)}},
    \label{eq:parallel_regulation}
\end{equation}
where $\mathcal{A}_i$ denotes the set of activator indices and $\mathcal{R}_i$ the set of repressor indices for gene~$i$; each factor in the first product is an increasing logistic (activation term), while each factor in the second product is a decreasing logistic (repression term).

The complete dynamical system governing the network is therefore
\begin{equation}
    \dot{x}_i
    \;=\;
    \kappa_i
    \!\left(
        \prod_{j \in \mathcal{A}_i}
            \frac{1}{1 + e^{-\lambda(x_j - \theta_{ij})}}
        \;\cdot\;
        \prod_{k \in \mathcal{R}_i}
            \frac{1}{1 + e^{-\lambda(\theta_{ik} - x_k)}}
    \right)
    - \gamma_i x_i.
    \label{eq:multi_gene_system}
\end{equation}
This formulation naturally accommodates any combination of activators and
repressors acting on each gene, with the thresholds $\theta_{ij}$ and
$\theta_{ik}$ allowing gene-specific and regulator-specific tuning of
sensitivity.  Biologically, this product form models independent binding
sites, where full activation requires all activators to be bound and all
repressors to be unbound, akin to AND logic.

\subsection{Well-Posedness, Forward Invariance, and Lipschitz Bounds}

The structural advantages of the logistic formulation translate into clean
analytical guarantees for the multi-gene system~\eqref{eq:multi_gene_system}.
We collect three properties---global well-posedness, forward invariance of
the biologically meaningful box, and a global Lipschitz estimate on the
right-hand side---into the following proposition.  The proof is brief but
the conclusions underwrite every numerical and control-theoretic claim that
follows.

\begin{proposition}[Global well-posedness and forward invariance]
\label{prop:wellposedness}
Consider the multi-gene logistic system~\eqref{eq:multi_gene_system}
on $\mathbb{R}^N$, with $\kappa_i,\gamma_i,\lambda,\theta_{ij},\theta_{ik}>0$.
Let
\[
  \mathbf{F}(\mathbf{x}) =
  \bigl(\kappa_1\Phi_1(\mathbf{x})-\gamma_1 x_1,\,
        \ldots,\,
        \kappa_N\Phi_N(\mathbf{x})-\gamma_N x_N\bigr),
\]
where each $\Phi_i\colon\mathbb{R}^N\to(0,1)$ is a product of increasing
and decreasing logistic factors as in~\eqref{eq:multi_gene_system}.
Then:
\begin{enumerate}
\item[(i)] $\mathbf{F}$ is of class $C^{\infty}(\mathbb{R}^N;\mathbb{R}^N)$
  and globally Lipschitz on $\mathbb{R}^N$ with constant
  \begin{equation}
    L \;=\; \max_{1\le i\le N}\Bigl(\gamma_i + \kappa_i\,\tfrac{\lambda}{4}
    \,\bigl(|\mathcal{A}_i|+|\mathcal{R}_i|\bigr)\Bigr).
    \label{eq:lipschitz_constant}
  \end{equation}
\item[(ii)] For every initial condition $\mathbf{x}(0)\in\mathbb{R}^N$
  there exists a unique solution $\mathbf{x}\in C^{1}([0,\infty);\mathbb{R}^N)$
  of~\eqref{eq:multi_gene_system}.
\item[(iii)] The closed box
  $\mathcal{B} = \prod_{i=1}^{N}[0,\kappa_i/\gamma_i]$ is forward
  invariant: if $\mathbf{x}(0)\in\mathcal{B}$, then
  $\mathbf{x}(t)\in\mathcal{B}$ for all $t\ge 0$, and in particular all
  components remain non-negative and bounded by $\kappa_i/\gamma_i$.
\end{enumerate}
\end{proposition}

\begin{proof}
\emph{(i)} Each logistic factor $f^\pm$ is in $C^{\infty}(\mathbb{R})$
with $|(f^\pm)'|\le\lambda/4$.  By the chain and product rules,
$\Phi_i$ is $C^{\infty}$, with
$|\partial\Phi_i/\partial x_j|\le\lambda/4$ for every regulator
$j\in\mathcal{A}_i\cup\mathcal{R}_i$ (and
$\partial\Phi_i/\partial x_j=0$ when $j$ does not regulate gene $i$),
because $\Phi_i\in(0,1)$ and only the single factor depending on $x_j$
contributes to the derivative.  Hence each component of $\mathbf{F}$ is
$C^{\infty}$ and the row sum of the Jacobian satisfies
\[
  \sum_{j=1}^{N}\Bigl|\frac{\partial F_i}{\partial x_j}\Bigr|
  \;\le\;
  \gamma_i \;+\; \kappa_i\,\tfrac{\lambda}{4}\,
  \bigl(|\mathcal{A}_i|+|\mathcal{R}_i|\bigr),
\]
which yields the global Lipschitz
estimate~\eqref{eq:lipschitz_constant} via the operator
$\infty$-norm bound on $\|\mathrm{D}\mathbf{F}\|$.

\emph{(ii)} Global Lipschitz continuity of $\mathbf{F}$ on $\mathbb{R}^N$
implies global existence and uniqueness by the Picard--Lindel\"of
theorem~\citep{hartman2002ordinary}.

\emph{(iii)} Forward invariance follows from inspection of the vector
field on $\partial\mathcal{B}$.  On the face $\{x_i=0\}$,
$\dot{x}_i=\kappa_i\Phi_i(\mathbf{x})>0$ strictly, since
$\Phi_i(\mathbf{x})\in(0,1)$ for every $\mathbf{x}\in\mathbb{R}^N$.
On the face $\{x_i=\kappa_i/\gamma_i\}$,
$\dot{x}_i=\kappa_i\Phi_i(\mathbf{x})-\kappa_i\le 0$, with equality only
in the unattainable limit $\Phi_i\equiv 1$.  By Nagumo's
theorem~\citep{blanchini2008set}, $\mathcal{B}$ is therefore positively
invariant.
\end{proof}

Two consequences of Proposition~\ref{prop:wellposedness} are worth
emphasising.  First, the bound~\eqref{eq:lipschitz_constant} is
\emph{linear} in $\lambda$ and in the in-degree
$|\mathcal{A}_i|+|\mathcal{R}_i|$ of each gene; for a network with a
fixed cooperativity matching $\lambda=n/\theta$ and bounded in-degree,
$L$ remains bounded uniformly in the network size $N$.  This is the
basis of the ``standard stability theory applies everywhere'' assertion
in Section~\ref{sec:hill_pathologies}: every adaptive solver enjoys
classical convergence guarantees of the form
$\|\mathbf{e}(t)\|\le\|\mathbf{e}(0)\|e^{Lt}$, with $L$ given
explicitly by~\eqref{eq:lipschitz_constant}.  The Hill formulation, as
detailed there, fails this hypothesis on the boundary $\{x_j=0\}$ and
the corresponding error bound diverges.  Second, forward invariance of
$\mathcal{B}$ holds for the \emph{exact} flow regardless of solver
behaviour; the strict positivity of $\Phi_i$ at $x_i=0$ is what
prevents permanent shutdown and mathematically underwrites the basal
production analysed in the companion paper~\citep{belgacem2026numerical}.

\subsection{Recovery of the Boolean Steady States}

Propositions~\ref{prop:demorgan} and~\ref{prop:wellposedness} guarantee that
the continuous model is well-posed and that its regulatory functions converge
to the Boolean ones. What a modeller ultimately needs, however, is the
stronger guarantee that the \emph{dynamical conclusions} of the Boolean
analysis survive the translation. The following theorem supplies it for the
steady states: every fixed point of the Boolean network reappears as a genuine,
exponentially stable equilibrium of the continuous logistic system.

\begin{theorem}[Boolean fixed-point recovery]
\label{thm:boolean_recovery}
Consider the Boolean-derived logistic system $\dot{x}_i=\kappa_i\Phi_i(\mathbf{x})-\gamma_i x_i$
of~\eqref{eq:ode_0}, in which $\Phi_i$ is the De~Morgan map of
Proposition~\ref{prop:demorgan} and every gene $j$ enters all regulatory
functions through a single threshold $\theta_j$ (as in the Traynard model of
Section~\ref{ex:traynard}), with $\kappa_j,\gamma_j>0$ and
\[
  0 \;<\; \theta_j \;<\; \kappa_j/\gamma_j ,
  \qquad j = 1,\ldots,N .
\]
Let $b\in\{0,1\}^{N}$ be a fixed point of the Boolean network, i.e.\
$b_i=\varphi_i(b)$ for all $i$, and let $\mathbf{v}(b)$ be the vertex of the box
$\mathcal{B}=\prod_i[0,\kappa_i/\gamma_i]$ with
$\mathbf{v}(b)_i=(\kappa_i/\gamma_i)\,b_i$.  Put
$\delta=\min_j\min\{\theta_j,\ \kappa_j/\gamma_j-\theta_j\}>0$.  Then there
exists $\lambda_{0}>0$ such that for every steepness $\lambda\ge\lambda_{0}$:
\begin{enumerate}
\item[(i)] \emph{(Recovery.)}  The system possesses an equilibrium
  $\mathbf{x}^{*}(\lambda)$ with
  $\|\mathbf{x}^{*}(\lambda)-\mathbf{v}(b)\|_{\infty}\le C\,e^{-\lambda\delta/2}$,
  the constant $C$ being independent of $\lambda$; in particular
  $x_i^{*}(\lambda)\to\kappa_i/\gamma_i$ when $b_i=1$ and
  $x_i^{*}(\lambda)\to 0$ when $b_i=0$ as $\lambda\to\infty$.
\item[(ii)] \emph{(Exponential stability.)}  The Jacobian at
  $\mathbf{x}^{*}(\lambda)$ is $J=-\operatorname{diag}(\gamma_i)+E(\lambda)$
  with $\|E(\lambda)\|_{\infty}\le C'\,\lambda\,e^{-\lambda\delta/2}$;
  consequently every eigenvalue of $J$ has real part at most
  $-\min_i\gamma_i+C'\lambda e^{-\lambda\delta/2}<0$, so
  $\mathbf{x}^{*}(\lambda)$ is locally exponentially stable.
\end{enumerate}
\end{theorem}

\begin{proof}
Write $\mathbf{G}^{\lambda}(\mathbf{x})_i=(\kappa_i/\gamma_i)\,\Phi_i(\mathbf{x})$,
so that the equilibria of~\eqref{eq:ode_0} are exactly the fixed points of
$\mathbf{G}^{\lambda}$.  Fix $r=\delta/2$ and work on the closed box
$B_r=\{\mathbf{x}:\|\mathbf{x}-\mathbf{v}(b)\|_{\infty}\le r\}$.  For
$\mathbf{x}\in B_r$ and any gene $j$, the coordinate $x_j$ lies within
$\delta/2$ of $\mathbf{v}(b)_j\in\{0,\kappa_j/\gamma_j\}$, hence at distance at
least $\delta/2$ from $\theta_j$ and on the same side of $\theta_j$ as in $b$;
so the discretisation $\operatorname{disc}(\mathbf{x})_j=\mathbf{1}_{\{x_j>\theta_j\}}$
equals $b_j$ throughout $B_r$.

\emph{Step 1 ($\mathbf{G}^{\lambda}$ is a contracting self-map of $B_r$).}
By Proposition~\ref{prop:demorgan}(ii), applied on $B_r$, which is
$\delta/2$-separated from every threshold hyperplane, there is a constant
$C_0$---the maximal literal count over the rules $\varphi_1,\ldots,\varphi_N$,
independent of $\lambda$---with
$\sup_{\mathbf{x}\in B_r}|\Phi_i(\mathbf{x})-\varphi_i(b)|\le C_0e^{-\lambda\delta/2}$;
since $\varphi_i(b)=b_i$ this gives
$|\mathbf{G}^{\lambda}(\mathbf{x})_i-\mathbf{v}(b)_i|
=(\kappa_i/\gamma_i)|\Phi_i(\mathbf{x})-b_i|\le(\kappa_i/\gamma_i)C_0e^{-\lambda\delta/2}$
on $B_r$, which is $\le r$ once $\lambda$ is large, so
$\mathbf{G}^{\lambda}(B_r)\subseteq B_r$.  On $B_r$ each logistic factor
$f^{\pm}(x_j,\theta_j)$ is within $e^{-\lambda\delta/2}$ of $\{0,1\}$, so its
derivative $\lambda f^{\pm}(1-f^{\pm})$ is bounded by $\lambda e^{-\lambda\delta/2}$;
since each partial derivative $\partial\Phi_i/\partial x_j$ is a sum of at most
$L_i$ terms (the literal count of $\varphi_i$), each differentiating a single
logistic factor and leaving the remaining factors in $[0,1]$, the row sum
satisfies $\sum_j|\partial\Phi_i/\partial x_j|\le L_i\,\lambda e^{-\lambda\delta/2}$.
Hence $\|\mathrm{D}\mathbf{G}^{\lambda}\|_{\infty}\le
\max_i(\kappa_i/\gamma_i)L_i\,\lambda e^{-\lambda\delta/2}$, which is $\le\tfrac12$
once $\lambda$ is large; $\mathbf{G}^{\lambda}$ is then a contraction on $B_r$.

\emph{Step 2 (existence and rate).}  By the Banach fixed-point theorem
$\mathbf{G}^{\lambda}$ has a unique fixed point $\mathbf{x}^{*}(\lambda)\in B_r$,
an equilibrium of~\eqref{eq:ode_0}.  The contraction estimate
$\|\mathbf{x}^{*}-\mathbf{v}(b)\|_{\infty}\le(1-\tfrac12)^{-1}
\|\mathbf{G}^{\lambda}(\mathbf{v}(b))-\mathbf{v}(b)\|_{\infty}$ together with the
bound of Step~1 gives
$\|\mathbf{x}^{*}(\lambda)-\mathbf{v}(b)\|_{\infty}\le
2\max_i(\kappa_i/\gamma_i)C_0\,e^{-\lambda\delta/2}=:C\,e^{-\lambda\delta/2}$,
proving~(i).

\emph{Step 3 (stability).}  The Jacobian of~\eqref{eq:ode_0} at
$\mathbf{x}^{*}(\lambda)$ is
$J=-\operatorname{diag}(\gamma_i)+\operatorname{diag}(\kappa_i)\,\mathrm{D}\boldsymbol{\Phi}=:-\operatorname{diag}(\gamma_i)+E(\lambda)$.
Since $\mathbf{x}^{*}(\lambda)\in B_r$, the row-sum bound of Step~1 yields
$\|E(\lambda)\|_{\infty}\le\max_i\kappa_i L_i\,\lambda e^{-\lambda\delta/2}=:C'\lambda e^{-\lambda\delta/2}$.
By Gershgorin's circle theorem every eigenvalue $\mu$ of $J$ satisfies
$|\mu+\gamma_i|\le\|E(\lambda)\|_{\infty}$ for some $i$, hence
$\operatorname{Re}\mu\le-\min_i\gamma_i+C'\lambda e^{-\lambda\delta/2}$, which
is negative once $\lambda$ is large enough that
$C'\lambda e^{-\lambda\delta/2}<\min_i\gamma_i$.  Taking $\lambda_0$ to be the
largest of the three lower bounds on $\lambda$ used above completes the
proof.
\end{proof}

Theorem~\ref{thm:boolean_recovery} is the precise sense in which the logistic
translation \emph{refines} the Boolean model rather than merely approximating
it: every Boolean steady state survives as a genuine, exponentially stable
equilibrium of the continuous dynamics, the discrete labels $0$ and $1$ acquire
the quantitative meaning of the basal and saturated concentrations $0$ and
$\kappa_i/\gamma_i$, and the rate $e^{-\lambda\delta/2}$ makes explicit how
response steepness controls the fidelity of the correspondence.  The hypothesis
$0<\theta_j<\kappa_j/\gamma_j$ is the natural one---it asks only that each
gene's regulatory threshold lie between its fully repressed and fully induced
levels.  Because the argument is local to each vertex $\mathbf{v}(b)$, distinct
Boolean fixed points lift to distinct equilibria, so a multistable Boolean
network yields a continuous model with at least as many stable steady states,
each labelled by the Boolean attractor it refines.  The theorem deliberately
claims no more than it proves: it does not assert the converse---the continuous
system may carry additional equilibria near the threshold hyperplanes---and it
does not address cyclic Boolean attractors, whose continuous counterparts
depend on the update schedule and lie beyond the present scope.

\subsection{Equivalence of Fixed-Weight and Weighted Formulations}
\label{sec:weight_equivalence}
Within our modelling framework, using fixed unit weights is formally equivalent
to incorporating explicit positive real-valued weights after appropriate
parameter rescaling.  We establish this equivalence separately for the
increasing and decreasing logistic functions, then combine them into the
product formulation as a concrete illustration.

The \emph{fixed-weight} increasing logistic for activator $x$ is $f^+(x, \theta, \lambda)= \frac{1}{1 + e^{-\lambda(x - \theta)}}$. Its inflection point is at $x = \theta$ and its maximum slope is $\lambda/4$. The \emph{weighted} increasing logistic introduces a strictly positive
interaction strength $w > 0$:
\begin{equation}
f^+( w x, \theta, \lambda)
  = \frac{1}{1 + e^{-\lambda(wx - \theta)}},
  \qquad w > 0.
\label{eq:fplus_weighted}
\end{equation}
Factoring the exponent, $\lambda(wx - \theta)= \lambda w \!\left(x - \frac{\theta}{w}\right)= \lambda'(x - \theta')$, where the rescaled parameters are
\begin{equation}
\lambda' = \lambda w > 0,
\qquad
\theta' = \frac{\theta}{w} > 0.
\label{eq:rescaling_plus}
\end{equation}
Hence,
\[
  f^+(w x, \theta, \lambda)
  = \frac{1}{1 + e^{-\lambda'(x - \theta')}}
  = f^+(x, \theta', \lambda'),
\]
which is identical in form to the fixed-weight function with parameters $(\lambda', \theta')$. The effective threshold $\theta' = \theta/w > 0$ is always positive and biologically interpretable as the activator concentration producing half-maximal activation ($\mathrm{EC}_{50}$).

Similarly, the \emph{fixed-weight} decreasing logistic for repressor $x$ is $f^-(x, \theta, \lambda)
  = \frac{1}{1 + e^{-\lambda(\theta - x)}}$. Its inflection point is at $x = \theta$ and its maximum slope magnitude is
$\lambda/4$. The \emph{weighted} decreasing logistic introduces a strictly positive interaction strength $w > 0$:
\begin{equation}
f^-(w x, \theta, \lambda)
  = \frac{1}{1 + e^{-\lambda(\theta - wx)}},
  \qquad w > 0.
\label{eq:fminus_weighted}
\end{equation}
Factoring the exponent, $\lambda(\theta - wx)  = \lambda w \!\left(\frac{\theta}{w} - x\right) = \lambda'(\theta' - x)$, with the same rescaling~\eqref{eq:rescaling_plus}.  Hence
\[
  f^-(w x, \theta, \lambda)
  = \frac{1}{1 + e^{-\lambda'(\theta' - x)}}
  = f^-(x, \theta', \lambda'),
\]
identical in form to the fixed-weight function with parameters $(\lambda', \theta')$. The effective threshold $\theta' = \theta/w > 0$ is again always positive, interpretable as the repressor concentration producing half-maximal inhibition ($\mathrm{IC}_{50}$).

\subsubsection{Product Formulation: A Concrete Illustration}

We now combine both logistic types into the product regulatory function.
Consider gene $i$ regulated by one activator $x_j$ with weight $w_{ij}>0$
and one repressor $x_k$ with weight $w_{ik}>0$. Introducing explicit weights $w_{ij}, w_{ik} > 0$:
\begin{equation}
\dot{x}_i = \kappa_i
\left(
  \frac{1}{1 + e^{-\lambda_{ij}(w_{ij}\,x_j - \theta_{ij})}}
  \cdot
  \frac{1}{1 + e^{-\lambda_{ik}(\theta_{ik} - w_{ik}\,x_k)}}
\right)
- \gamma_i x_i.
\label{eq:product_weighted}
\end{equation}

Applying the rescaling~\eqref{eq:rescaling_plus} to each factor
independently,
\[
  \lambda_{ij}(w_{ij}\,x_j - \theta_{ij})
    = \lambda_{ij}'(x_j - \theta_{ij}'),
  \qquad
  \lambda_{ik}(\theta_{ik} - w_{ik}\,x_k)
    = \lambda_{ik}'(\theta_{ik}' - x_k),
\]
where
\begin{equation}
\lambda_{ij}' = \lambda_{ij} w_{ij},
\quad
\theta_{ij}' = \frac{\theta_{ij}}{w_{ij}},
\quad
\lambda_{ik}' = \lambda_{ik} w_{ik},
\quad
\theta_{ik}' = \frac{\theta_{ik}}{w_{ik}}.
\label{eq:rescaling_product}
\end{equation}
Substituting into~\eqref{eq:product_weighted} recovers exactly the
fixed-weight system with parameters: $(\lambda_{ij}', \theta_{ij}', \lambda_{ik}', \theta_{ik}')$.

The general weighted formulation for gene $i$ with activator set
$\mathcal{A}_i$ and repressor set $\mathcal{R}_i$ is therefore
{\small
\begin{equation}
\dot{x}_i = \kappa_i
\!\left(
  \prod_{j \in \mathcal{A}_i}
    \frac{1}{1 + e^{-\lambda_{ij}(w_{ij}\,x_j - \theta_{ij})}}
  \;\cdot\;
  \prod_{k \in \mathcal{R}_i}
    \frac{1}{1 + e^{-\lambda_{ik}(\theta_{ik} - w_{ik}\,x_k)}}
\right)
- \gamma_i x_i,
\label{eq:multi_gene_system_weighted}
\end{equation}
}
with $w_{ij},\, w_{ik} > 0$,
which is equivalent to the fixed-weight system via the rescaling
$\lambda_{ij}' = \lambda_{ij}w_{ij}$,
$\theta_{ij}' = \theta_{ij}/w_{ij}$ for all $j \in \mathcal{A}_i$, and
$\lambda_{ik}' = \lambda_{ik}w_{ik}$,
$\theta_{ik}' = \theta_{ik}/w_{ik}$ for all $k \in \mathcal{R}_i$. The two systems produce \emph{identical trajectories} after parameter
estimation, since both parameter sets
$\{(\lambda_{ij}, \theta_{ij})\}$ and $\{(\lambda_{ij}', \theta_{ij}')\}$
are estimated from the same experimental data.  The weighted formulation
redistributes steepness and threshold information across two parameters
rather than one, but the invariant quantities governing the sigmoid
shape---the effective steepness $\lambda' = \lambda w$ and the effective
threshold $\theta' = \theta/w$---are the same in both cases.
One may therefore choose whichever parameterisation best suits the
available data, interpretability requirements, or computational
convenience, with the guarantee that the underlying biological dynamics
are identical.

\section{Application: The Traynard Mammalian Cell-Cycle Network}
\label{ex:traynard}

To illustrate the framework on a concrete biologically grounded network, we consider the Traynard mammalian cell-cycle model~\citep{Traynard2016}, a well-established Boolean network comprising eleven genes: $\mathrm{Cdc20}$, $\mathrm{Cdh1}$, $\mathrm{CycA}$, $\mathrm{CycB}$, $\mathrm{CycD}$, $\mathrm{CycE}$, $\mathrm{E2F}$, $\mathrm{p27}$, $\mathrm{Rb}$, $\mathrm{Skp2}$, and $\mathrm{UbcH10}$.

\subsection{Boolean Network}
The regulatory rules governing each gene are the following propositional
formulae:
{\footnotesize
\begin{align*}
\varphi_{\mathrm{Cdc20}}
    &= \mathrm{CycB},\\
\varphi_{\mathrm{Cdh1}}
    &= (\neg\mathrm{CycA}\wedge\neg\mathrm{CycB})\vee\mathrm{p27},\\
\varphi_{\mathrm{CycA}}
    &= (\neg\mathrm{Cdc20}\wedge\neg\mathrm{Cdh1}\wedge\mathrm{CycA})
       \vee(\neg\mathrm{Cdc20}\wedge\neg\mathrm{Cdh1}
           \wedge\mathrm{E2F}\wedge\neg\mathrm{Rb})\\
    &\quad
       \vee(\mathrm{CycA}\wedge\neg\mathrm{UbcH10})
       \vee(\mathrm{E2F}\wedge\neg\mathrm{Rb}\wedge\neg\mathrm{UbcH10}),\\
\varphi_{\mathrm{CycB}}
    &= (\neg\mathrm{Cdc20}\wedge\neg\mathrm{Cdh1})
       \vee(\neg\mathrm{Cdh1}\wedge\neg\mathrm{UbcH10}),\\
\varphi_{\mathrm{CycE}}
    &= \mathrm{E2F}\wedge\neg\mathrm{Rb},\\
\varphi_{\mathrm{E2F}}
    &= (\neg\mathrm{Cdc20}\wedge\neg\mathrm{CycA}\wedge\neg\mathrm{Rb})
       \vee(\neg\mathrm{Cdc20}\wedge\mathrm{p27}\wedge\neg\mathrm{Rb})\\
    &\quad
       \vee(\neg\mathrm{CycA}\wedge\neg\mathrm{CycB}\wedge\neg\mathrm{Rb})
       \vee(\neg\mathrm{CycB}\wedge\mathrm{p27}\wedge\neg\mathrm{Rb}),\\
\varphi_{\mathrm{p27}}
    &= (\neg\mathrm{CycA}\wedge\neg\mathrm{CycB}
        \wedge\neg\mathrm{CycD}\wedge\neg\mathrm{CycE})
       \vee(\neg\mathrm{CycA}\wedge\neg\mathrm{CycB}
           \wedge\neg\mathrm{CycD}\wedge\mathrm{p27})\\
    &\quad
       \vee(\neg\mathrm{CycB}\wedge\neg\mathrm{CycD}
           \wedge\neg\mathrm{CycE}\wedge\mathrm{p27})
       \vee(\neg\mathrm{CycD}\wedge\neg\mathrm{Skp2}),\\
\varphi_{\mathrm{Rb}}
    &= (\neg\mathrm{CycA}\wedge\neg\mathrm{CycB}
        \wedge\neg\mathrm{CycD}\wedge\neg\mathrm{CycE})
       \vee(\neg\mathrm{CycA}\wedge\neg\mathrm{CycD}\wedge\mathrm{p27})\\
    &\quad
       \vee(\neg\mathrm{CycB}\wedge\neg\mathrm{CycD}\wedge\mathrm{p27})
       \vee(\neg\mathrm{CycD}\wedge\neg\mathrm{CycE}\wedge\mathrm{p27}),\\
\varphi_{\mathrm{Skp2}}
    &= \neg\mathrm{Cdh1}\vee\neg\mathrm{Rb},\\
\varphi_{\mathrm{UbcH10}}
    &= (\mathrm{Cdc20}\wedge\mathrm{UbcH10})
       \vee\neg\mathrm{Cdh1}
       \vee(\mathrm{CycA}\wedge\mathrm{UbcH10})
       \vee(\mathrm{CycB}\wedge\mathrm{UbcH10}),\\
\varphi_{\mathrm{CycD}}
    &= \mathrm{CycD}.
\end{align*}
}

\subsection{Continuous Logistic ODE System}
Applying the map $\Phi$ \eqref{eq:ode_0}--\eqref{eq:demorgan_0} to each
Boolean rule, each gene $x_i$ evolves according to
$\dot{x}_i = \kappa_i\,\Phi_i(\mathbf{x})-\gamma_i\,x_i$.
With the steepness matching $\lambda_i = n/\theta_i$ and a shared
cooperativity $n = 4$, the eleven continuous regulatory functions are
{\small
\begin{align}
\Phi_{\mathrm{Cdc20}}
    &= f^{+}(\mathrm{CycB},\,\theta_{\mathrm{CycB}}),
\label{eq:Cdc20}\\
\Phi_{\mathrm{Cdh1}}
    &= 1 - \bigl(1 - f^{-}(\mathrm{CycA},\theta_{\mathrm{CycA}})\,
                      f^{-}(\mathrm{CycB},\theta_{\mathrm{CycB}})\bigr)\,
           f^{-}(\mathrm{p27},\theta_{\mathrm{p27}}),
\label{eq:Cdh1}\\
\Phi_{\mathrm{CycA}}
    &= 1 - \bigl(1-f^{-}_{20}f^{-}_{h1}f^{+}_{A}\bigr)
           \bigl(1-f^{-}_{20}f^{-}_{h1}f^{+}_{E2F}f^{-}_{Rb}\bigr)
           \bigl(1-f^{+}_{A}f^{-}_{H10}\bigr)
           \bigl(1-f^{+}_{E2F}f^{-}_{Rb}f^{-}_{H10}\bigr),
\label{eq:CycA}\\
\Phi_{\mathrm{CycB}}
    &= 1 - \bigl(1-f^{-}_{20}f^{-}_{h1}\bigr)
           \bigl(1-f^{-}_{h1}f^{-}_{H10}\bigr),
\label{eq:CycB}\\
\Phi_{\mathrm{CycE}}
    &= f^{-}(\mathrm{Rb},\theta_{\mathrm{Rb}})\,
       f^{+}(\mathrm{E2F},\theta_{\mathrm{E2F}}),
\label{eq:CycE}\\
\Phi_{\mathrm{E2F}}
    &= 1 - \bigl(1-f^{-}_{20}f^{-}_{A}f^{-}_{Rb}\bigr)
           \bigl(1-f^{-}_{20}f^{+}_{27}f^{-}_{Rb}\bigr)
           \bigl(1-f^{-}_{A}f^{-}_{B}f^{-}_{Rb}\bigr)
           \bigl(1-f^{-}_{B}f^{+}_{27}f^{-}_{Rb}\bigr),
\label{eq:E2F}\\
\Phi_{\mathrm{p27}}
    &= 1 - \bigl(1-f^{-}_{A}f^{-}_{B}f^{-}_{D}f^{-}_{E}\bigr)
           \bigl(1-f^{-}_{A}f^{-}_{B}f^{-}_{D}f^{+}_{27}\bigr)
           \bigl(1-f^{-}_{B}f^{-}_{D}f^{-}_{E}f^{+}_{27}\bigr)
           \bigl(1-f^{-}_{D}f^{-}_{Sk}\bigr),
\label{eq:p27}\\
\Phi_{\mathrm{Rb}}
    &= 1 - \bigl(1-f^{-}_{A}f^{-}_{B}f^{-}_{D}f^{-}_{E}\bigr)
           \bigl(1-f^{-}_{A}f^{-}_{D}f^{+}_{27}\bigr)
           \bigl(1-f^{-}_{B}f^{-}_{D}f^{+}_{27}\bigr)
           \bigl(1-f^{-}_{D}f^{-}_{E}f^{+}_{27}\bigr),
\label{eq:Rb}\\
\Phi_{\mathrm{Skp2}}
    &= 1 - f^{+}(\mathrm{Cdh1},\theta_{\mathrm{Cdh1}})\,
           f^{+}(\mathrm{Rb},\theta_{\mathrm{Rb}}),
\label{eq:Skp2}\\
\Phi_{\mathrm{UbcH10}}
    &= 1 - f^{+}_{h1}
           \bigl(1-f^{+}_{20}f^{+}_{H10}\bigr)
           \bigl(1-f^{+}_{A}f^{+}_{H10}\bigr)
           \bigl(1-f^{+}_{B}f^{+}_{H10}\bigr),
\label{eq:UbcH10}\\
\Phi_{\mathrm{CycD}}
    &= f^{+}(\mathrm{CycD},\theta_{\mathrm{CycD}}),
\label{eq:CycD}
\end{align}
}
where we use the shorthand $f^{\pm}_{i} \equiv f^{\pm}(x_i,\theta_i)$
with subscripts:
$20$~=~$\mathrm{Cdc20}$,
$h1$~=~$\mathrm{Cdh1}$,
$A$~=~$\mathrm{CycA}$,
$B$~=~$\mathrm{CycB}$,
$D$~=~$\mathrm{CycD}$,
$E$~=~$\mathrm{CycE}$,
$27$~=~$\mathrm{p27}$,
$Rb$~=~$\mathrm{Rb}$,
$Sk$~=~$\mathrm{Skp2}$,
$H10$~=~$\mathrm{UbcH10}$,
$E2F$~=~$\mathrm{E2F}$.
Each factor $f^{-}(\mathrm{p27}) = 1 - f^{+}(\mathrm{p27})$ implements the
complement structure of Boolean negation~\eqref{eq:not_logic}, and the
four-factor De~Morgan products in
\eqref{eq:CycA}--\eqref{eq:Rb} implement the multi-clause OR
formula~\eqref{eq:demorgan_0}.

\subsection{Parameters}
Production rates $\kappa_i$, degradation rates $\gamma_i$, and thresholds
$\theta_i$ are drawn from $\mathcal{U}(50,100)$, $\mathcal{U}(0.25,2)$,
and $\mathcal{U}(10,20)$ respectively, then rounded to two decimal places.
The fixed realisation used throughout is reported in
Table~\ref{tab:traynard_params}.

\begin{table}[ht]
\centering
\caption{Kinetic parameters for the Traynard cell-cycle ODE system.
         Cooperativity $n=4$ (shared).  Initial conditions $x_i(0)$ are
         sampled from $\mathcal{U}(0,100)$ and rounded to two decimal
         places.}
\label{tab:traynard_params}
\small
\begin{tabular}{lcccc}
\toprule
Gene $i$ & $\kappa_i$ & $\gamma_i$ & $\theta_i$ & $x_i(0)$ \\
\midrule
$\mathrm{Cdc20}$   & 74.33 & 0.70 & 19.18 &  0.54 \\
$\mathrm{Cdh1}$    & 61.00 & 1.45 & 11.84 & 38.68 \\
$\mathrm{CycA}$    & 79.17 & 1.74 & 19.29 & 96.61 \\
$\mathrm{CycB}$    & 76.70 & 0.94 & 18.84 & 64.56 \\
$\mathrm{CycE}$    & 91.06 & 0.58 & 18.90 & 40.97 \\
$\mathrm{E2F}$     & 56.50 & 0.58 & 17.30 & 69.56 \\
$\mathrm{p27}$     & 68.79 & 0.68 & 14.69 &  5.55 \\
$\mathrm{Rb}$      & 53.20 & 1.24 & 11.73 & 32.23 \\
$\mathrm{Skp2}$    & 66.65 & 0.43 & 12.95 & 63.64 \\
$\mathrm{UbcH10}$  & 73.51 & 1.95 & 12.05 & 45.80 \\
$\mathrm{CycD}$    & 64.75 & 0.76 & 19.89 & 58.36 \\
\bottomrule
\end{tabular}
\end{table}

\noindent
The box $\mathcal{B}=\prod_i[0,\kappa_i/\gamma_i]$ is forward invariant and
globally attracting (Proposition~\ref{prop:wellposedness}).  Two of the sampled
initial values, $x_{\mathrm{CycA}}(0)=96.61$ and $x_{\mathrm{UbcH10}}(0)=45.80$,
exceed their respective ceilings $\kappa_i/\gamma_i=45.50$ and $37.70$; these two
components therefore decrease monotonically into $\mathcal{B}$ over the first few
time units, after which $x_i(t)\le\kappa_i/\gamma_i$ holds for every gene.

\subsection{Numerical Simulation}
All translations from the Boolean network to a continuous ODE system can be
automated: the Boolean network \texttt{FB} is translated automatically into a
continuous ODE system using our function \texttt{BooleanToODESys}.
The ODE system~\eqref{eq:Cdc20}--\eqref{eq:CycD} is integrated numerically
over $t \in [0,60]$ using \texttt{NDSolve} in Mathematica (default
adaptive step-size control, default error tolerances).  The solver completes
the integration without any warnings, and all 11 state variables remain
non-negative throughout.  As noted above, the two components launched above
their ceiling (CycA and UbcH10) relax into the forward-invariant box
$\prod_i[0,\kappa_i/\gamma_i]$ within the first few time units, after which every
variable satisfies $x_i(t)\le\kappa_i/\gamma_i$, in agreement with
Proposition~\ref{prop:wellposedness}.  The resulting trajectories are shown in
Figure~\ref{fig:traynard_sim}.

\begin{figure}[t]
\centering
\includegraphics[width=\linewidth]{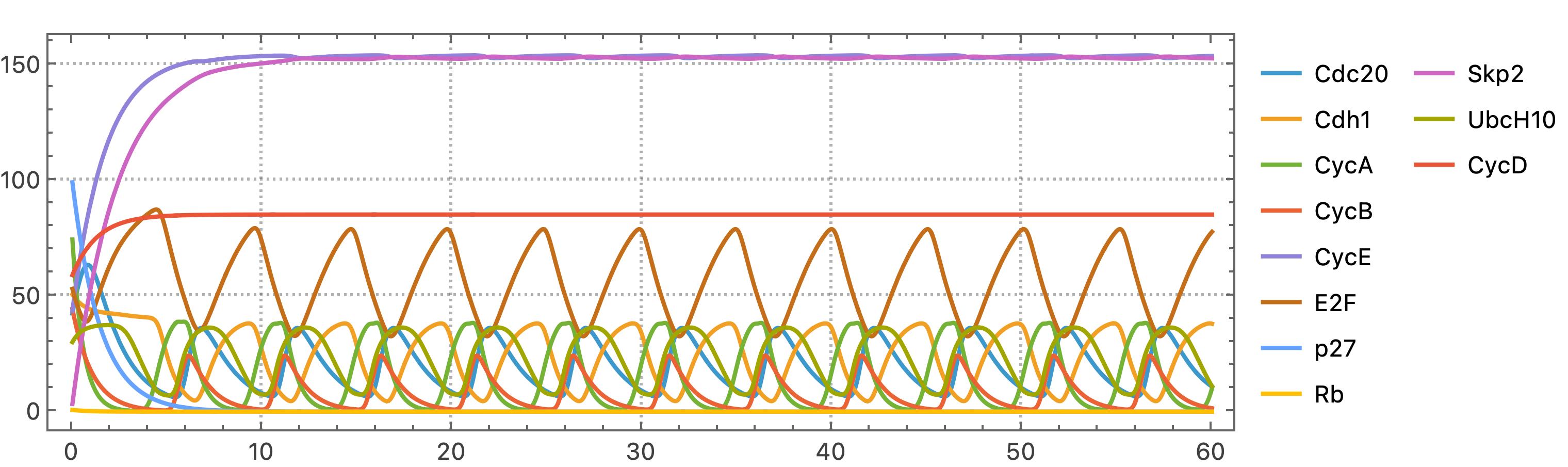}
\caption{Temporal evolution of the 11-gene Traynard cell-cycle logistic ODE
         system over $t \in [0,60]$, starting from the initial conditions
         in Table~\ref{tab:traynard_params}.  Parameters: $n = 4$;
         $\kappa_i$, $\gamma_i$, $\theta_i$ as listed in
         Table~\ref{tab:traynard_params}.  All variables remain
         non-negative and, after a brief transient, bounded by
         $\kappa_i/\gamma_i$.  Because the initial condition lies in the
         proliferative regime (CycD active), the trajectories settle onto a
         sustained limit cycle (period ${\approx}5$): the cell-cycle regulators
         Cdc20, Cdh1, CycA, CycB, E2F and UbcH10 oscillate, CycD, CycE and Skp2
         stay high, and the quiescence markers p27 and Rb stay off, reproducing
         the cyclic attractor of the Boolean model.}
\label{fig:traynard_sim}
\end{figure}

\noindent
The simulation illustrates the key advantages of the logistic formulation
in a biologically grounded large-scale network.  The always-positive basal
production rate
$f^+(0,\theta_i,n/\theta_i) = 1/(1+e^{n}) > 0$
prevents any gene from being permanently trapped in the off-state, in
contrast to Hill-function models where $h^+(0) = 0$ would make the zero
state absorbing for genes whose sole activator is initially absent.  Because the chosen initial condition has
$x_{\mathrm{CycD}}(0)>\theta_{\mathrm{CycD}}$ while the rule
$\varphi_{\mathrm{CycD}}=\mathrm{CycD}$ keeps CycD active, the network sits in
its proliferative regime, in which the Boolean model has \emph{no} fixed point
but a cyclic attractor; the continuous trajectories accordingly settle onto a
\emph{sustained limit cycle} rather than an equilibrium, with the cell-cycle
regulators Cdc20, Cdh1, CycA, CycB, E2F and UbcH10 oscillating, CycD, CycE and
Skp2 staying high, and the quiescence markers p27 and Rb staying off
(Figure~\ref{fig:traynard_sim}).  As is typical of a smooth relaxation of a
Boolean cycle, the continuous limit cycle need not reproduce every variable's
full on--off excursion---here CycE remains high because E2F never falls below
its threshold along the orbit---but the qualitative proliferative programme is
faithfully recovered.  This is the network-scale,
cyclic-attractor counterpart of the fixed-point correspondence of
Theorem~\ref{thm:boolean_recovery}: that theorem rigorously recovers the Boolean
\emph{steady states}---illustrated in closed form by the toggle switch of
Section~\ref{ex:toggle}---and deliberately leaves cyclic attractors outside its
scope, whereas the Traynard example shows empirically that the purely structural
translation also transports the Boolean \emph{cyclic} attractor to a continuous
limit cycle.  This sustained oscillation is moreover not in tension with
Theorem~\ref{thm:no_hopf}, which rules out limit cycles only for the planar
two-gene motif via Bendixson's criterion in $\mathbb{R}^2$---a two-dimensional
obstruction that does not apply in dimension $N\ge 3$, where the same delay-free
logistic framework readily supports oscillation.  The
bounded production $\Phi_i(\mathbf{x}) \in [0,1]$ enforces the asymptotic
ceiling $x_i(t) \leq \kappa_i/\gamma_i$ on the forward-invariant box, while the
$C^{\infty}$ regularity
of the right-hand side allows the solver to take large adaptive time steps
without step-size blow-up.  Together, these properties make the logistic
formulation immediately deployable for the simulation, attractor
identification, and control of Boolean-derived ODE systems at the scale of
realistic cell-cycle models.

\section{Comparison with the Samuilik Weighted-Sum Formulation}
\label{sec:comparison_literature}

A widely used alternative formulation for gene regulatory network
modelling aggregates all regulatory inputs into a single weighted sum
passed through one increasing logistic function per gene
\citep{samuilik2022mathematical}:

\begin{equation}
\dot{x}_i = \kappa_i \cdot
  \frac{1}{1 + e^{-\mu_i\!\left(\sum_{j=1}^n w_{ij} x_j - \theta_i\right)}}
  - \gamma_i x_i,
\quad w_{ij} \in \mathbb{R}.
\label{eq:samuilik_model}
\end{equation}
Here $\mu_i$ is a gene-specific steepness parameter, $\theta_i$ is a single shared threshold for gene~$i$ (rather than regulator-specific thresholds $\theta_{ij}$), and the weights $w_{ij}$ are real-valued and signed: positive for activation, negative for repression, both fed through the same increasing logistic. Subsequent studies have perpetuated this formulation~\citep{samuilik2022mathematical,sadyrbaev2021modelling,kozlovska2024search,samuilik2022genetic,sadyrbaev2023coexistence,somathilaka2023revealing,kozlovska2025modeling}.
Yet this compromise carries structural costs that are not merely
aesthetic: as we demonstrate rigorously below for AND, OR, and NOR
gates, the formulation produces repression functions whose critical
points lie at negative concentrations outside the physical domain,
thresholds that scale with network size and lose all correspondence
with measurable molecular quantities, and sigmoid shapes that remain
nearly flat and biologically inert throughout the admissible
concentration range $x \geq 0$. These are not parameter artefacts
that can be corrected by recalibration; they are unavoidable
consequences of the architectural choice to encode regulatory
direction through weight signs inside a single increasing sigmoid.

Our product-of-logistics framework resolves all three defects by
construction. The purpose of this section is to make those differences
precise, fair, and complete. We compare the two formulations on three
canonical gate types: AND (multi-input activation), OR (independent
activation), and a pure-inhibitor gate (multiple repressors, no
activators). In each case we show that our logistic approach preserves clear biological semantics, independently identifiable parameters, and network-size-invariant thresholds, while the Samuilik weighted-sum formulation does not.

\subsection{Structural Distinctions Between the Two Formulations}
\label{subsec:structural_distinctions}

The two formulations differ along four interrelated axes that we summarise here before pursuing the gate-by-gate analysis. First, with respect to the encoding of regulatory direction, a negative weight inside an increasing sigmoid (the Samuilik convention) and a positive weight inside a decreasing sigmoid (our convention) are \emph{not} equivalent representations of repression. A negative weight inside an increasing sigmoid places the sigmoid's critical point at $x_c = \theta/w < 0$ (with $\theta>0$ and $w<0$), strictly outside the physically admissible domain $x \geq 0$, so the function never undergoes its sigmoidal transition over any biologically realisable concentration and remains nearly flat and close to zero throughout. The decreasing-sigmoid convention places the critical point at $x_c = \theta/w > 0$, a positive, biologically interpretable inhibition midpoint ($\mathrm{IC}_{50}$) that can be measured directly from dose-response data. This is a structural pathology of the Samuilik encoding, not a consequence of any particular parameter choice, and it persists for every repressor in every network regardless of size.

Second, with respect to the regulatory architecture, the Samuilik model aggregates all regulatory inputs into a single weighted sum $\sum w_{ij}x_j$ passed through one increasing logistic per gene, modelling additive or competitive effects, whereas our product-of-logistics approach multiplies individual sigmoidal terms for each regulator. The latter naturally captures multiplicative (AND) interactions, in which multiple conditions must be satisfied simultaneously, as well as OR interactions via the De~Morgan product formula, while keeping each regulator's contribution analytically separable.

Third, with respect to the threshold structure, the Samuilik model uses a single threshold $\theta_i$ pooled across all regulators of gene~$i$, prescribed as $\theta_i = \sum_j w_{ij}/2$. This threshold grows linearly with the number of activators, becomes negative for pure-inhibitor gates, and depends on the expression levels of all co-regulators simultaneously, making it impossible to determine from any single-regulator experiment. Our model assigns a regulator-specific threshold $\theta_{ij}/w_{ij}$ to every interaction independently; this quantity is context-free, network-size-invariant, and directly identifiable from individual dose-response measurements.

Fourth, with respect to biological interpretability, each factor in our product represents the probability-like occupancy of an independent binding site, and the effective threshold $\theta_{ij}/w_{ij}>0$ maps directly to a dissociation constant $K_d$ or half-maximal effective concentration $\mathrm{EC}_{50}$, measurable from single-regulator dose-response experiments independently of all other regulators in the network. The Samuilik model's composite threshold $\theta_i$ loses this correspondence entirely: it carries no interpretable link to any binding affinity or half-maximal concentration, and cannot be validated against experimental data without simultaneously fitting all weights and the shared threshold in a single ill-conditioned optimisation.

\subsection{Dynamical Comparison on the Two-Gene Oscillator}
\label{subsec:oscillator_comparison}

The structural distinctions of Section~\ref{subsec:structural_distinctions}
have a direct dynamical counterpart, which the two-gene negative-feedback
oscillator of Section~\ref{ex:oscillator} exposes in the sharpest possible
form. We model the \emph{same} circuit---gene~1 repressed by gene~2, gene~2
activated by gene~1---three ways: with Hill functions, with our
product-of-logistics functions (already introduced in
Section~\ref{ex:oscillator}), and with the Samuilik weighted-sum
formulation~\eqref{eq:samuilik_model}; we then integrate each from the common
initial condition $x_1(0)=x_2(0)=1$ exactly as in
Figure~\ref{fig:Oscillateur_original}.

All three are calibrated to the same regulatory geometry. The \emph{logistic}
model (this work) is equation~\eqref{eq:oscillator}; written out, with the
parameters of Figure~\ref{fig:Oscillateur_original} ($\lambda=3$, $\kappa_1=3$,
$\gamma_1=0.25$, $\kappa_2=4$, $\gamma_2=0.5$, $\theta_1=4$, $\theta_2=3$), it
reads
\begin{equation}
\begin{aligned}
\dot{x}_1 &= \kappa_1\, f^-(x_2,\theta_2,\lambda) - \gamma_1 x_1
          = \kappa_1\,\frac{1}{1+e^{\lambda(x_2-\theta_2)}} - \gamma_1 x_1, \\[3pt]
\dot{x}_2 &= \kappa_2\, f^+(x_1,\theta_1,\lambda) - \gamma_2 x_2
          = \kappa_2\,\frac{1}{1+e^{-\lambda(x_1-\theta_1)}} - \gamma_2 x_2 .
\end{aligned}
\label{eq:osc_logistic}
\end{equation}
The \emph{Hill} model replaces each logistic by the Hill function matched at its
threshold through the slope rule $\lambda=n/\theta$, i.e.\ $n=\lambda\theta$:
\begin{equation}
\dot{x}_1 = \kappa_1\,\frac{\theta_2^{\,n_2}}{x_2^{\,n_2}+\theta_2^{\,n_2}}
            -\gamma_1 x_1,
\qquad
\dot{x}_2 = \kappa_2\,\frac{x_1^{\,n_1}}{x_1^{\,n_1}+\theta_1^{\,n_1}}
            -\gamma_2 x_2,
\label{eq:osc_hill}
\end{equation}
with $n_1=\lambda\theta_1=12$ for the activation and $n_2=\lambda\theta_2=9$
for the repression.

The \emph{Samuilik} model routes both regulatory inputs through a single
increasing logistic per gene acting on a signed weighted sum, with canonical
weights $w_{12}=-1$ (gene~2 represses gene~1) and $w_{21}=+1$ (gene~1 activates
gene~2) and matched steepness $\mu=\lambda=3$:
\begin{equation}
\dot{x}_1 = \kappa_1\,\frac{1}{1+e^{-\mu(w_{12}x_2-\theta_{12})}}-\gamma_1 x_1,
\qquad
\dot{x}_2 = \kappa_2\,\frac{1}{1+e^{-\mu(w_{21}x_1-\theta_{21})}}-\gamma_2 x_2 ,
\label{eq:osc_samuilik}
\end{equation}
where $\theta_{12}$ and $\theta_{21}$ denote the thresholds of the repression
and the activation respectively. Crucially, the weighted-sum formulation does
not fix these thresholds uniquely, and the two natural conventions yield
markedly different dynamics.

\emph{(i)~Measured thresholds.} Assigning each interaction the dissociation
constant at which the corresponding logistic and Hill sigmoid actually
switches---$\theta_{12}=3$ for the repression of gene~1 by gene~2 and
$\theta_{21}=4$ for the activation of gene~2 by gene~1, i.e.\ exactly the
$\theta_2$ and $\theta_1$ used in~\eqref{eq:osc_logistic}--\eqref{eq:osc_hill}
---equation~\eqref{eq:osc_samuilik} becomes
\begin{equation}
\dot{x}_1 = \kappa_1\,\frac{1}{1+e^{-\mu(-x_2-3)}}-\gamma_1 x_1,
\qquad
\dot{x}_2 = \kappa_2\,\frac{1}{1+e^{-\mu(x_1-4)}}-\gamma_2 x_2 .
\label{eq:osc_samuilik_bio}
\end{equation}

\emph{(ii)~Prescribed thresholds.} The standard shared-threshold prescription
$\theta_i=\tfrac12\sum_j
w_{ij}$~\citep{samuilik2022mathematical,kozlovska2022models} instead gives
$\theta_{12}=w_{12}/2=-\tfrac12$ and $\theta_{21}=w_{21}/2=+\tfrac12$.

Figure~\ref{fig:osc_comparison} reports formulation~(ii) alongside the Hill and
logistic models; the measured-threshold variant~(i) is analysed in
Remark~\ref{rem:samuilik_collapse} and Figure~\ref{fig:osc_comparison}(d).

\begin{figure}[t]
\centering
\includegraphics[width=\linewidth]{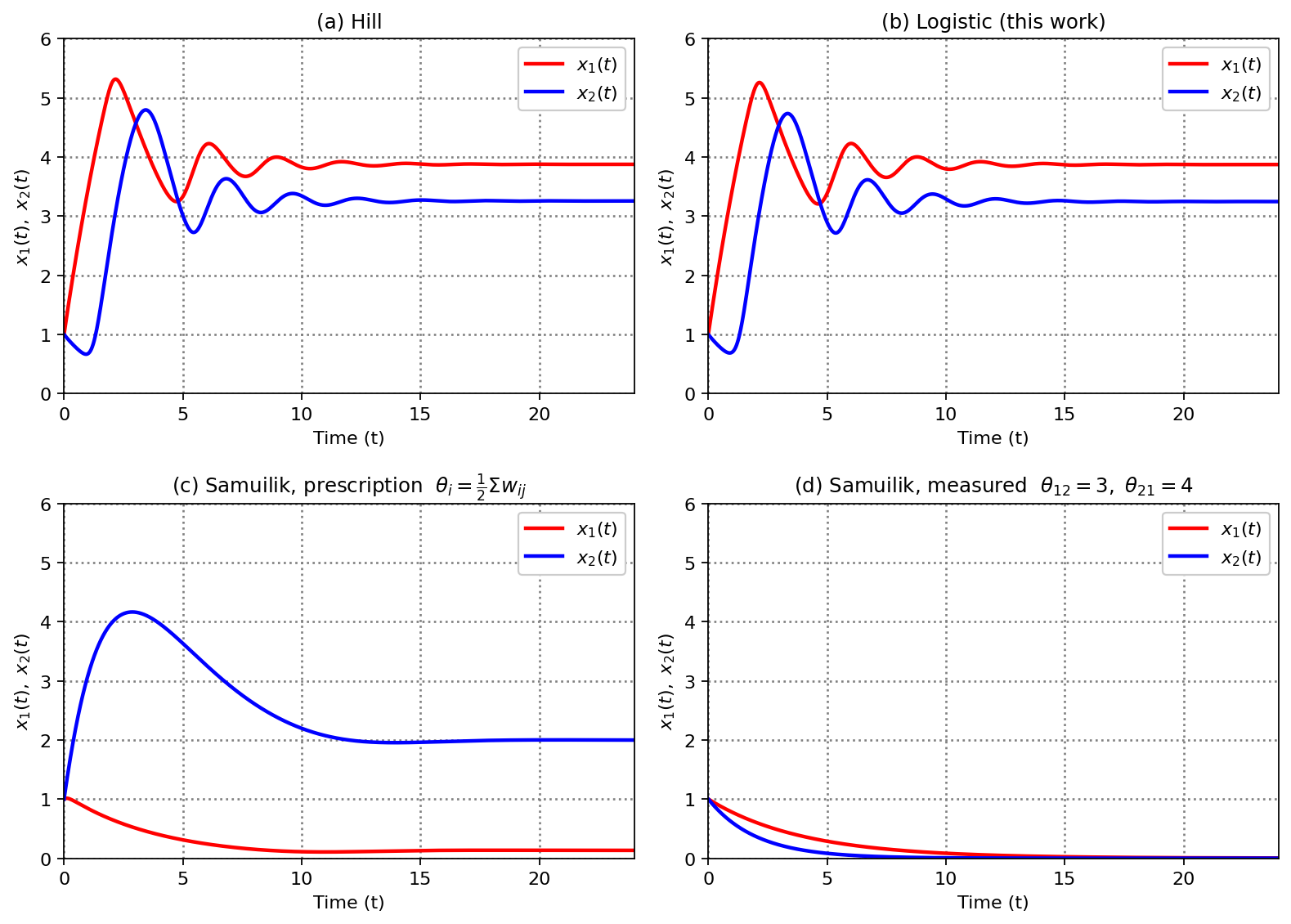}
\caption{The two-gene oscillator of Section~\ref{ex:oscillator}
($x_1(0)=x_2(0)=1$, integration as in Figure~\ref{fig:Oscillateur_original})
under four formulations of the \emph{same} circuit, all drawn on a common scale.
\textbf{(a)}~Hill functions~\eqref{eq:osc_hill}, slope-matched via
$n=\lambda\theta$ ($n_1=12$, $n_2=9$).
\textbf{(b)}~Product-of-logistics (this work),
equation~\eqref{eq:osc_logistic}, $\lambda=3$.
\textbf{(c)}~Samuilik weighted sum~\eqref{eq:osc_samuilik} under the prescribed
shared thresholds $\theta_i=\tfrac12\sum_j w_{ij}$ (i.e.\
$\theta_{12}=-\tfrac12$, $\theta_{21}=+\tfrac12$), $\mu=\lambda$.
\textbf{(d)}~Samuilik weighted sum~\eqref{eq:osc_samuilik_bio} under the
biologically measured thresholds $\theta_{12}=3$, $\theta_{21}=4$, $\mu=\lambda$.
Panels~(a) and~(b) are visually indistinguishable, converging to the same
equilibrium $(3.87,3.25)$ with the same damped overshoot. Neither Samuilik
parameterisation reproduces this: panel~(c) reaches a grossly displaced steady
state $(0.13,2.00)$ with gene~1 held essentially silent, while panel~(d), using
the biologically faithful thresholds, collapses to the trivial state $(0,0)$ as
the negative-weight repression saturates (Remark~\ref{rem:samuilik_collapse}).}
\label{fig:osc_comparison}
\end{figure}

\begin{table}[t]
\centering
\caption{Equilibria and Jacobian spectra of the two-gene oscillator under the
three formulations (and, in the last row, the Samuilik model under the measured
thresholds $\theta_{12}=3$, $\theta_{21}=4$ of~\eqref{eq:osc_samuilik_bio}
rather than the prescription; see Remark~\ref{rem:samuilik_collapse}). The correct equilibrium is
$(x_1^\ast,x_2^\ast)\approx(3.87,3.25)$. All non-degenerate cases share
$\operatorname{tr}J=-(\gamma_1+\gamma_2)=-0.75$, so the real part is fixed at
$-0.375$ and the formulations differ in the oscillation frequency (imaginary
part) and in the location of the steady state.}
\label{tab:osc_comparison}
\footnotesize
\setlength{\tabcolsep}{5pt}
\resizebox{\linewidth}{!}{%
\begin{tabular}{lccc}
\toprule
Formulation & Equilibrium $(x_1^\ast,x_2^\ast)$ & Eigenvalues of $J$ & Behaviour \\
\midrule
Hill ($n_1{=}12$, $n_2{=}9$)        & $(3.88,\,3.26)$ & $-0.375\pm2.32\,i$  & damped oscillation \\
Logistic (this work)                & $(3.87,\,3.25)$ & $-0.375\pm2.38\,i$  & damped oscillation \\
Samuilik ($\theta_i=\tfrac12\sum w$)& $(0.13,\,2.00)$ & $-0.375\pm0.454\,i$ & weakly damped; gene~1 silent \\
Samuilik (measured $\theta$)        & $(0,\,0)$       & $-0.25,\,-0.50$     & monotone collapse \\
\bottomrule
\end{tabular}}
\end{table}

Figure~\ref{fig:osc_comparison} and Table~\ref{tab:osc_comparison} make three
points. First, the Hill and logistic trajectories are visually
indistinguishable: both spiral into the same equilibrium---$(3.88,3.26)$ for
Hill, $(3.87,3.25)$ for the logistic---with the same damped overshoot and nearly
identical spectra ($-0.375\pm2.32\,i$ versus $-0.375\pm2.38\,i$). At the level of
a complete trajectory this confirms that, once slopes are matched by
$\lambda=n/\theta$, the logistic form reproduces Hill dynamics while dispensing
with fractional exponents.

Second, \emph{none} of the three formulations produces a sustained oscillation:
each is a planar negative-feedback loop with
$\operatorname{tr}J=-(\gamma_1+\gamma_2)<0$, so by Theorem~\ref{thm:no_hopf}
every equilibrium is a globally attracting focus or node and no limit cycle can
exist. The distinguishing feature between the formulations is therefore not the
presence or absence of oscillation---the planar topology already forbids
it---but \emph{where the steady state sits} and \emph{how vigorously the loop
rings} en route. The Samuilik shortfall is consequently a quantitative and
structural one, not a failure to oscillate.

Third, on both surviving criteria the Samuilik formulation misrepresents the
circuit. Its prescribed shared thresholds place the activation and repression
sigmoid midpoints at the concentration $x_c=\theta_i/w_{ij}=\tfrac12$
(Section~\ref{subsec:structural_distinctions}), far below the biological
thresholds $\theta_1=4$ and $\theta_2=3$ at which the logistic and Hill sigmoids
actually switch. The steady state is consequently displaced to $(0.13,2.00)$
---gene~1 is held essentially silent, a thirty-fold error against the correct
$x_1^\ast\approx3.87$---and the damped overshoot all but vanishes: the
oscillation frequency drops from $2.38$ to $0.454$
(Table~\ref{tab:osc_comparison}), because both sigmoids operate deep in their
saturated tails where the feedback gain $\mu f(1-f)$ is negligible. The
weighted-sum encoding does not merely re-scale parameters; it relocates the
regulatory transition out of the biologically relevant window and thereby
distorts both the equilibrium and the transient. This is, moreover, the more
forgiving of the two threshold conventions: the biologically faithful
choice~(i) of equation~\eqref{eq:osc_samuilik_bio} fares worse still,
extinguishing the circuit entirely, as we now show.

\begin{remark}[Parameter ambiguity: neither threshold convention reproduces the oscillator]
\label{rem:samuilik_collapse}
Formulation~(i) exposes the deeper difficulty. Equipping the repression of
gene~1 with its biologically measured threshold $\theta_{12}=3$ and the negative
weight $w_{12}=-1$ places the repression midpoint at
$x_c=\theta_{12}/w_{12}=-3<0$, outside the physical domain---the static midpoint
pathology of Section~\ref{subsec:structural_distinctions} (cf.\
Figure~\ref{fig:comparison}). The repression sigmoid
in~\eqref{eq:osc_samuilik_bio} then never switches for $x_2\ge0$: gene~1's
production is pinned at
$\Phi_1\le 1/(1+e^{\mu\theta_{12}})=1/(1+e^{9})\approx1.2\times10^{-4}$, so the
repressor can never release, gene~1 cannot drive its activation target, and the
whole loop collapses to the trivial state $(x_1,x_2)\to(0,0)$
(Figure~\ref{fig:osc_comparison}(d)). The two threshold conventions thus bracket
the same negative conclusion: the prescription~(ii) $\theta_i=\tfrac12\sum_j
w_{ij}$ leaves gene~1 nearly silent at a displaced equilibrium $(0.13,2.00)$
(Figure~\ref{fig:osc_comparison}c), while the biologically faithful choice~(i)
extinguishes the circuit altogether (Figure~\ref{fig:osc_comparison}(d)). In
both, the negative-weight repression never operates within the admissible
concentration range $x_2\ge0$---the precise sense in which, under the
weighted-sum encoding, the inhibition is never exercised. By contrast the
product-of-logistics repression $f^-(x_2,\theta_2,\lambda)$ places its midpoint
at the positive, measurable $\theta_2=3$ and reproduces the circuit faithfully
(Figure~\ref{fig:osc_comparison}b).
\end{remark}

\subsection{The AND Gate: Single-Repressor and Mixed Regulation}
\label{subsec:and_comparison}

A logically coherent evaluation requires comparing the two implementations
of the \emph{same} logical operation with weights included on both sides.
To place the comparison on an equal structural footing, we present our
framework in its \emph{weighted} form, which uses real-valued interaction
strengths $w_{ij}, w_{ik} > 0$ for each regulator. The weighted
product-of-logistics model for gene~$i$ with activator index set
$\mathcal{A}_i$ and repressor index set $\mathcal{R}_i$ is
{\small
\begin{equation}
\dot{x}_i = \kappa_i
\!\left(
  \prod_{j\in\mathcal{A}_i}
    \frac{1}{1+e^{-\lambda\bigl(w_{ij}x_j-\theta_{ij}\bigr)}}
  \;\cdot\;
  \prod_{k\in\mathcal{R}_i}
    \frac{1}{1+e^{-\lambda\bigl(\theta_{ik}-w_{ik}x_k\bigr)}}
\right)
- \gamma_i x_i.
\label{eq:multi_gene_weighted_uniform}
\end{equation}
}
with $w_{ij},\,w_{ik} > 0$.
As established in Section~\ref{sec:weight_equivalence}, this weighted
formulation is parameter-rescaling equivalent to the fixed-weight
model~\eqref{eq:multi_gene_system} via $\lambda' = \lambda w$ and
$\theta' = \theta/w$; the two systems produce identical trajectories after
parameter estimation, differing only in how steepness and threshold
information is distributed across parameters.

This model differs fundamentally from the alternative formulation of
\citet{samuilik2022mathematical}. The structural
difference between the two weighted formulations is transparent. Both
deploy real-valued weights; the distinction lies in \emph{how regulatory
direction is encoded}. In the Samuilik
model~\eqref{eq:samuilik_model}, direction is encoded by the \emph{sign}
of $w_{ij}$ inside a single increasing sigmoid, so repression requires
$w_{ij}<0$. In our model~\eqref{eq:multi_gene_weighted_uniform}, all weights
are \emph{strictly positive} ($w_{ij},w_{ik}>0$), and direction is encoded
by the \emph{functional form}: an increasing logistic for each activator, a
decreasing logistic for each repressor. This seemingly small architectural
choice has far-reaching consequences for biological realism, as we now
demonstrate. For completeness, note that~\eqref{eq:multi_gene_weighted_uniform} can equivalently be written with negative weights for repressors,
matching the sign convention of the Samuilik formulation:
\begin{equation}
\begin{aligned}
\dot{x}_i = \kappa_i
\!\left(
  \prod_{j\in\mathcal{A}_i}
    \frac{1}{1+e^{-\lambda\bigl(w_{ij}x_j-\theta_{ij}\bigr)}}
  \;\cdot\;
  \prod_{k\in\mathcal{R}_i}
    \frac{1}{1+e^{-\lambda\bigl(w_{ik}x_k + \theta_{ik}\bigr)}}
\right)
- \gamma_i x_i,\\[4pt]
w_{ij}>0,\quad w_{ik} < 0.
\end{aligned}
\label{eq:multi_gene_system_weighted_2}
\end{equation}

\subsubsection{Illustrating Repression: The Midpoint Pathology}

Consider a single repressor $x \geq 0$ acting on gene $i$. In our framework, with strictly positive weight $w > 0$, the repression term is the weighted decreasing logistic
\begin{equation}
f^-(wx,\theta,\lambda) \;=\; \frac{1}{1+e^{-\lambda(\theta - wx)}}.
\label{eq:our_repression_weighted}
\end{equation}
Its midpoint (where the output equals $1/2$ and the slope is steepest) satisfies
\[
  \theta - w x_c = 0
  \;\Longrightarrow\;
  x_c \;=\; \frac{\theta}{w} \;>\; 0,
\]
which is always \emph{positive} (since $\theta > 0$ and $w > 0$) and is biologically interpretable as the repressor concentration producing half-maximal inhibition (an effective $\mathrm{IC}_{50}$).

By contrast, the Samuilik formulation encodes repression by feeding a \emph{negative} weight $w<0$ into an increasing sigmoid:
\begin{equation}
f_2^-(x,\theta,\mu) \;=\; \frac{1}{1+e^{-\mu(wx-\theta)}}, \qquad w < 0.
\label{eq:samuilik_repression}
\end{equation}
Its midpoint satisfies
\[
  w x_c - \theta = 0
  \;\Longrightarrow\;
  x_c \;=\; \frac{\theta}{w} \;<\; 0
  \quad(\text{since } \theta > 0,\; w < 0),
\]
which is \emph{negative} and lies entirely outside the biologically relevant domain $x \geq 0$. Consequently, $f_2^-$ never undergoes its sigmoidal transition over any biologically realisable concentration; it remains nearly flat and close to zero throughout $x \geq 0$.

Concretely, with $w=-1$, $\mu = \lambda = n/\theta = 4/3$, $\theta = 3$:
\[
  f_2^-(x) = \frac{1}{1+\exp\!\left(\tfrac{4}{3}(x+3)\right)}.
\]
For $x \geq 0$ this yields
$f_2^-(x) \leq \tfrac{1}{1+e^{4}} \approx 0.018$,
so the function remains nearly flat (never approaching~$1$) over the entire biologically relevant domain. This is not a parameter artefact; it is a structural consequence of placing the critical point at $x_c = -3$. By contrast, our $f^-$ correctly places the sigmoid midpoint at the positive concentration $\theta/w$ and closely approximates Hill repression behaviour, as illustrated in Fig.~\ref{fig:comparison}.

\begin{figure}[t]
\centering
\includegraphics[height=0.45\linewidth, width=0.9\linewidth, keepaspectratio]
{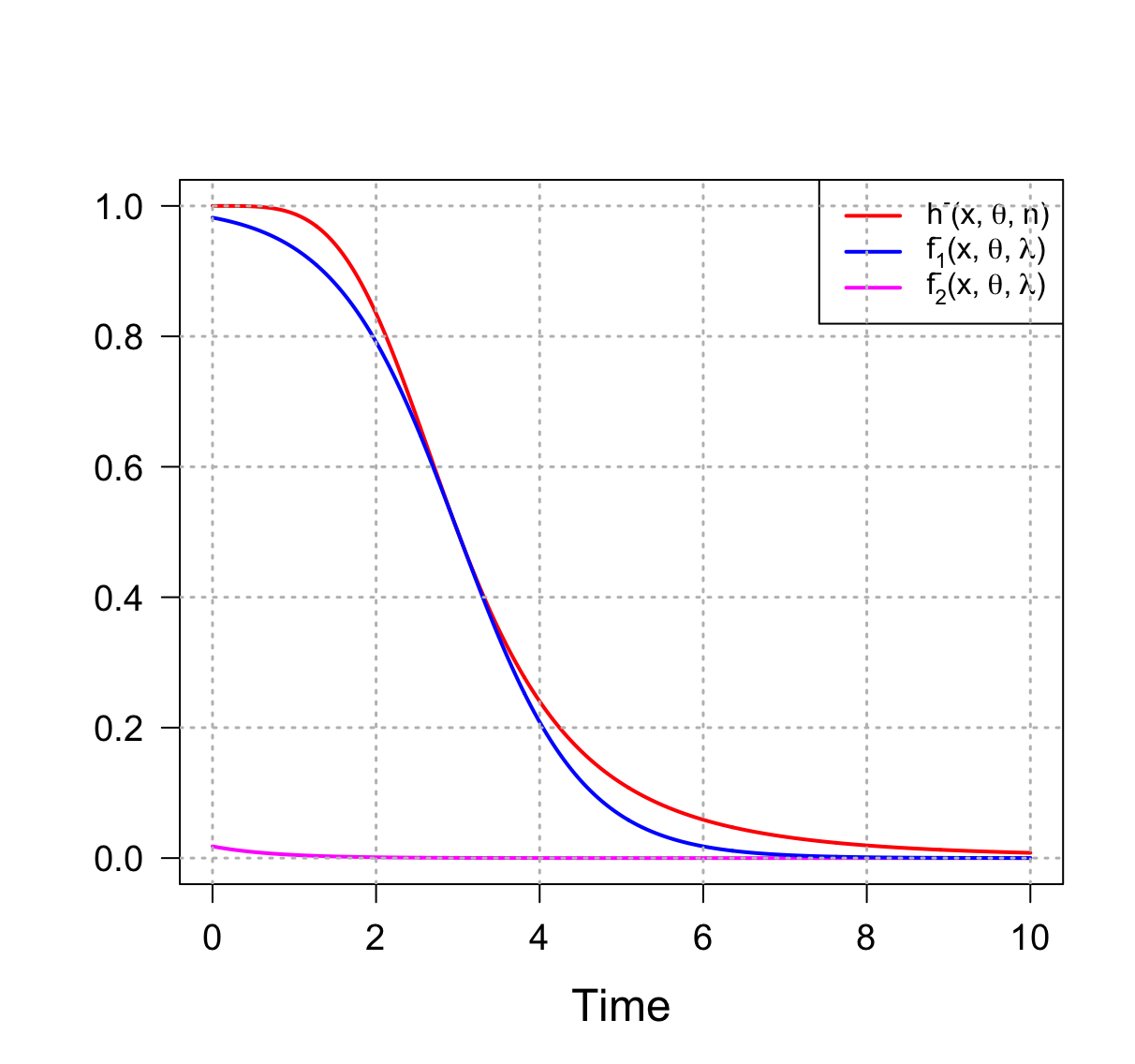}
\caption{Comparison of the decreasing logistic functions $f^-(wx,\theta,\lambda) = 1/(1+e^{-\lambda(\theta-wx)})$ (our model, with $w=1$) and the Samuilik form $f_2^-(x,\theta,\mu) = 1/(1+\exp(-\mu(wx-\theta)))$ (with $w=-1$), alongside the decreasing Hill function $h^-(x,\theta,n) = \theta^n/(x^n+\theta^n)$. Parameters: $n=4$, $\theta=3$, $\lambda = n/\theta \approx 1.333$, $\mu = \lambda$. The midpoint of our $f^-$ is at $x_c = \theta/w = 3 > 0$ (biologically meaningful), while that of $f_2^-$ is at $x_c = \theta/w = -3 < 0$ (outside the physical domain), causing $f_2^-$ to remain nearly zero throughout $x \geq 0$.}
\label{fig:comparison}
\end{figure}

\subsubsection{Mixed Activation and Repression}
Consider a gene regulated by one activator $x_1$ ($w_{i1}>0$) and one
repressor $x_2$ ($w_{i2}>0$). In our weighted
framework~\eqref{eq:multi_gene_weighted_uniform} the regulatory function is
\begin{equation}
f_i^{(1)}(x_1,x_2)
  = \frac{1}{1+e^{-\lambda(w_{i1}x_1 - \theta_{i1})}}
  \cdot
  \frac{1}{1+e^{-\lambda(\theta_{i2} - w_{i2}x_2)}},
  \qquad w_{i1}, w_{i2} > 0.
\label{eq:our_mixed_weighted}
\end{equation}
The effective activation threshold $\theta_{i1}/w_{i1}>0$ and the effective
repression threshold $\theta_{i2}/w_{i2}>0$ can be determined independently
from single-regulator dose-response experiments, and both are positive.
Setting $w_{i1}=w_{i2}=1$ recovers the fixed-weight
form~\eqref{eq:parallel_regulation}.

In the Samuilik model the same circuit is represented by
\begin{equation}
f_i^{(2)}(x_1,x_2)
  = \frac{1}{1+e^{-\mu_i(w_{i1}x_1 + w_{i2}x_2 - \theta_i)}},
  \qquad w_{i1} > 0,\; w_{i2} < 0,
\label{eq:samuilik_mixed}
\end{equation}
with the threshold prescribed as
$\theta_i = (w_{i1}+w_{i2})/2$~\citep{kozlovska2022models,samuilik2022mathematical}.
With the canonical choice $w_{i1}=+1$, $w_{i2}=-1$:
\[
  \theta_i = \frac{1+(-1)}{2} = 0.
\]
A threshold of zero lacks any biological interpretation: it does not
correspond to a measurable inflection point in the dose-response of either
regulator. From a theoretical standpoint, half-maximal production $\kappa_i/2$
is reached when $\sum_j w_{ij}x_j = \theta_i$, so $\theta_i$ should match a
characteristic concentration at which the regulators collectively place the
sigmoid at its inflection point. Yet with $\theta_i = 0$ this balance is
achieved only when the activating and repressing inputs exactly cancel, not at
any characteristic molecular concentration. The prescription
$\theta_i = \sum_j w_{ij}/2$ is mathematically convenient but biologically
unmotivated; the resulting $\theta_i = 0$ offers no interpretable link to
binding affinities or half-maximal concentrations.

The fundamental ambiguity remains: \emph{how should the weights $w_{ij}$ be
assigned?} In our weighted
formulation~\eqref{eq:multi_gene_weighted_uniform} no such ambiguity arises,
because all weights are positive and the effective thresholds
$\theta_{ij}/w_{ij}$ are anchored directly to experimental measurements,
regardless of how many regulators are present.

The shared-threshold architecture of~\eqref{eq:samuilik_model} requires simultaneous fitting of weights and threshold for any AND gate, an ill-conditioned inverse problem whenever two or more regulators are present. By contrast, the product-of-logistics formulation~\eqref{eq:multi_gene_weighted_uniform} assigns each interaction its own effective threshold $\theta_{ij}/w_{ij}>0$, directly identifiable from single-regulator binding-affinity measurements ($K_d$ via EMSA or SPR) or from dose-response characterisation ($\mathrm{EC}_{50}$ from reporter assays). Unlike Hill functions, where non-integer exponents are empirical fitting parameters lacking mechanistic interpretation~\citep{bintu2005transcriptional}, the logistic parameters $\lambda$, $w_{ij}$, $\theta_{ij}$---or equivalently the rescaled $\lambda' = \lambda w_{ij}$ and $\theta'_{ij} = \theta_{ij}/w_{ij}$---all possess direct experimental correlates that enable independent validation.

\subsection{The OR Gate}
\label{subsec:or_comparison}

We now compare the two formulations on \emph{OR} logic: gene~$i$ is
activated by $m$ independent activators $x_1,\ldots,x_m$, where the
presence of \emph{any one} sufficiently expressed activator is sufficient
to drive expression. The Boolean formula is $x_1 \vee \cdots \vee x_m$.

In the Samuilik framework~\citep{samuilik2022mathematical}, activators enter
through \emph{positive} weights $w_{ij}>0$ inside the single shared
increasing logistic:
\begin{equation}
\Phi_{\mathrm{WS}}^{\mathrm{OR}}(x_1,\ldots,x_m)
  = \frac{1}{1+\exp\!\Bigl(-\mu_i\!\Bigl(
      \sum_{j=1}^m w_{ij}\,x_j - \theta_i\Bigr)\Bigr)},
  \quad w_{ij} > 0,
\label{eq:ws_or}
\end{equation}
with the standard threshold
prescription~\citep{samuilik2022mathematical,kozlovska2022models}
\begin{equation}
  \theta_i = \frac{\sum_{j=1}^m w_{ij}}{2} > 0.
\label{eq:ws_or_threshold}
\end{equation}

In our framework, all weights are strictly positive ($w_{ij}>0$) and the
OR gate is derived from the De~Morgan product
formula~\eqref{eq:demorgan_0}:
\begin{equation}
\Phi_{\mathrm{DM}}^{\mathrm{OR}}(x_1,\ldots,x_m)
  = 1 - \prod_{j=1}^{m} \!\bigl(1-f^+(w_{ij}\,x_j,\,\theta_{ij},\,\lambda)\bigr)
  = 1 - \prod_{j=1}^{m} f^-(w_{ij}\,x_j,\,\theta_{ij},\,\lambda),
\label{eq:dm_or}
\end{equation}
where each factor $f^-(w_{ij}x_j,\theta_{ij},\lambda)\in(0,1)$ represents
the \emph{non-activated} fraction for activator~$j$. The formula is the
probability that at least one of $m$ independent activation events occurs;
the complete dynamical system is $\dot{x}_i = \kappa_i\,\Phi(\mathbf{x}) -
\gamma_i x_i$.

\subsubsection{Cross-Input Interaction Structure}

Define $\mathcal{I}_{12}=\partial^2\Phi/(\partial x_1\,\partial x_2)$.

\smallskip
\noindent\textit{Weighted-sum OR.}
Writing $s = \sum_j w_{ij}x_j - \theta_i$ and
$\sigma(s) = (1+e^{-\mu_i s})^{-1}$:
\begin{align*}
\frac{\partial\Phi_{\mathrm{WS}}^{\mathrm{OR}}}{\partial x_1}
  &= \mu_i w_{i1}\,\sigma(s)\bigl(1-\sigma(s)\bigr), \\
\frac{\partial^2\Phi_{\mathrm{WS}}^{\mathrm{OR}}}{\partial x_1\,\partial x_2}
  &= \mu_i^2\,w_{i1}w_{i2}\,\sigma(s)\bigl(1-\sigma(s)\bigr)
     \bigl(1-2\sigma(s)\bigr).
\end{align*}
Since $w_{i1},w_{i2}>0$ in the OR-gate case, $w_{i1}w_{i2}>0$, so the
cross-derivative has the sign of $(1-2\sigma(s))$, which changes sign at
the inflection surface $\{s=0\}$,
i.e.\ $\bigl\{\sum_j w_{ij}x_j = \theta_i\bigr\}$. The two activators
combine \emph{linearly} in~$s$ (since $\partial^2 s/\partial
x_1\partial x_2 = 0$); the apparent interaction is an artefact of the
shared sigmoid's curvature and encodes no genuine biological cooperativity
between the two activators.

\smallskip
\noindent\textit{Weighted De~Morgan OR.}
Write $p_j = f^+(w_{ij}x_j,\theta_{ij},\lambda)$ and
$q_j = 1-p_j = f^-(w_{ij}x_j,\theta_{ij},\lambda)$, so that
$\Phi_{\mathrm{DM}}^{\mathrm{OR}} = 1 - \prod_{j=1}^m q_j$.
For $m=2$:
\begin{align*}
\frac{\partial\Phi_{\mathrm{DM}}^{\mathrm{OR}}}{\partial x_1}
  &= \lambda w_{i1}\,p_1(1-p_1)\,q_2, \\
\frac{\partial^2\Phi_{\mathrm{DM}}^{\mathrm{OR}}}{\partial x_1\,\partial x_2}
  &= -\lambda^2\,w_{i1}w_{i2}\,p_1(1-p_1)\,p_2(1-p_2).
\end{align*}
For general $m \geq 2$, the same calculation yields
\[
  \frac{\partial^2\Phi_{\mathrm{DM}}^{\mathrm{OR}}}{\partial x_1\,\partial x_2}
  = -\lambda^2\,w_{i1}w_{i2}\,p_1(1-p_1)\,p_2(1-p_2)
    \prod_{j=3}^{m} q_j,
\]
where the additional factor $\prod_{j=3}^m q_j \in (0,1)$ further attenuates
the cross-input sensitivity as more activators are present.
In both cases this cross-derivative is \emph{always negative} (since
$p_j\in(0,1)$ and $w_{ij}>0$), encoding a genuine biological interaction:
when activator~$x_2$ is already saturating ($q_2\approx 0$), the marginal
contribution of~$x_1$ is attenuated---the gene cannot be more than fully
activated. This diminishing-returns behaviour is the natural property of OR
logic with independently acting activators.

\subsubsection{Threshold Semantics and Identifiability}

\noindent\textit{Weighted-sum OR.}
The inflection of $\Phi_{\mathrm{WS}}^{\mathrm{OR}}$ with respect to $x_1$
(others fixed) satisfies:
\[
  x_1^* = \frac{\theta_i - \sum_{j\neq 1}w_{ij}x_j}{w_{i1}}.
\]
For equal-weight activators ($w_{ij}=+1$, $\theta_i=m/2$) this reduces to
\[
  x_1^* = \frac{m}{2} - \sum_{j\neq 1}x_j.
\]
This effective threshold \emph{depends on the expression levels of all other activators} and exhibits a three-regime structure. When the other activators are absent ($\sum_{j\neq 1}x_j=0$), one obtains $x_1^*=m/2>0$: the threshold lies inside the physical domain, but grows linearly with $m$, a biologically unrealistic scaling that requires an ever-increasing activator concentration for half-maximal effect as more activators are added to the network. When the other activators reach the critical combined level $\sum_{j\neq 1}x_j=m/2$, the threshold collapses to $x_1^*=0$, at the very boundary of the physical domain. Finally, when the other activators are expressed at moderate-to-high levels ($\sum_{j\neq 1}x_j>m/2$), the threshold becomes $x_1^*<0$, strictly outside the physical domain; under typical biological conditions with multiple co-expressed activators, this regime is generically reached as $m$ grows, placing the per-variable inflection point at a biologically inaccessible negative concentration. In all three regimes the threshold fails independent identifiability. The
shared threshold $\theta_i=m/2$, which grows linearly with network size,
carries no direct biological interpretation as a single-activator
half-activation concentration.

\smallskip
\noindent\textit{Weighted De~Morgan OR.}
The inflection of $\Phi_{\mathrm{DM}}^{\mathrm{OR}}$ with respect to~$x_1$
(others fixed) is determined by the inflection of
$f^+(w_{i1}x_1,\theta_{i1},\lambda)$:
\[
  x_1^* = \frac{\theta_{i1}}{w_{i1}} > 0,
  \quad\text{independent of all other activators.}
\]
The effective threshold $\theta_{i1}/w_{i1}>0$ is the concentration of
activator~$x_1$ at which its individual activating contribution reaches
half-maximum, regardless of what other activators are doing. It is directly
identifiable from a single-activator dose-response experiment and maps to
the $\mathrm{EC}_{50}$ or dissociation constant $K_d$ of the activator--promoter
interaction.

\subsubsection{Single-Activator Sufficiency}

A defining property of OR logic is that a single sufficiently expressed
activator alone can fully activate the target gene, independent of the
states of other activators.

\smallskip
\noindent\textit{Weighted-sum OR.}
As $x_1\to\infty$ with $x_j=0$ for all $j\neq 1$:
\[
  \Phi_{\mathrm{WS}}^{\mathrm{OR}}
  \;\to\;
  \frac{1}{1+e^{-\mu_i(w_{i1}x_1 - \theta_i)}}
  \;\to\; 1,
\]
since $w_{i1}>0$. A single strongly expressed activator drives the gate
toward unity---qualitatively correct. However, under the equal-weight
prescription ($w_{ij}=+1$, $\theta_i=m/2$), the single-activator
half-activation threshold is $x_1^*=m/2$: it grows proportionally with
network size~$m$, becoming increasingly biologically unrealistic as the
number of potential activators grows.

\smallskip
\noindent\textit{Weighted De~Morgan OR.}
As $x_1\to\infty$ with $x_j=0$ for all $j\neq 1$:
\[
  \Phi_{\mathrm{DM}}^{\mathrm{OR}}
  = 1 - f^-(w_{i1}x_1,\theta_{i1},\lambda)\cdot\prod_{j\neq 1}
    f^-(0,\theta_{ij},\lambda)
  \;\to\;
  1 - 0 \cdot \prod_{j\neq 1} f^-(0,\theta_{ij},\lambda)
  = 1.
\]
Full activation is achieved by a single saturated activator regardless
of~$m$, and the activation threshold remains
$x_1\approx\theta_{i1}/w_{i1}$, independent of how many other potential
activators exist in the network.

\subsubsection{Basal Expression When All Activators Are Absent}

\noindent\textit{Weighted-sum OR.}
When all activators are absent ($x_j=0$ for all~$j$):
\[
  \Phi_{\mathrm{WS}}^{\mathrm{OR}}\big|_{x_j=0}
  = \frac{1}{1+e^{\mu_i\theta_i}}.
\]
Since $\theta_i=m/2>0$, the exponent $\mu_i\theta_i>0$ and
$\Phi_{\mathrm{WS}}^{\mathrm{OR}}|_{x_j=0}<1/2$. For large $m$, the
equal-weight basal level $1/(1+e^{\mu_i m/2})\to 0$: the gene approaches
silence when all activators are absent, qualitatively appropriate for an OR
gate. However, the basal level is coupled to the shared threshold and
cannot be tuned independently of the activation threshold.

\smallskip
\noindent\textit{Weighted De~Morgan OR.}
When all activators are absent ($x_j=0$ for all~$j$):
\[
  \Phi_{\mathrm{DM}}^{\mathrm{OR}}\big|_{x_j=0}
  = 1 - \prod_{j=1}^m f^-(0,\theta_{ij},\lambda)
  = 1 - \prod_{j=1}^m \frac{1}{1+e^{-\lambda\theta_{ij}}}.
\]
Each factor
$f^-(0,\theta_{ij},\lambda)=1/(1+e^{-\lambda\theta_{ij}})\in(1/2,\,1)$
for $\theta_{ij}>0$. Since each factor is strictly less than~$1$, increasing
the number $m$ of activators decreases the product $\prod_j f^-(0,\theta_{ij},\lambda)$ and
hence raises the basal OR level $1-\prod_j f^-$.
Each absent activator contributes a small
residual activation $f^+(0,\theta_{ij},\lambda)>0$ through its sigmoid
tail, and the OR formula $1-\prod_j(1-f^+)$ compounds these residual
contributions. The magnitude of this effect is controlled by the
products $\lambda\theta_{ij}$: when $\lambda\theta_{ij}$ is moderate (say $\lambda\theta_{ij}\approx 4$,
typical for $n=4$ Hill matching), each factor lies near~$0.98$, and the basal level grows
slowly with $m$ but is non-negligible already for ten or so activators.
This is a structural consequence of
the logistic function's strict positivity at zero input, not a parameter
artefact, and reflects the biologically appropriate emergent property that
a gene regulated by many independent activators has slightly elevated basal
expression due to the cumulative residual partial activation from each
activator's sigmoid tail. Each factor can be independently tuned
via~$\theta_{ij}$ without affecting the others.

\subsection{The NOR Gate (Pure Inhibitor Case)}
\label{subsec:nor_comparison}

We analyse the scenario in which gene~$i$ is regulated by $m$ repressors
$x_1,\ldots,x_m$ exclusively, with \textbf{no activators present}. The
Boolean logic is therefore \textbf{NOR}: gene~$i$ is active if and only if
\emph{none} of the $m$ repressors is sufficiently expressed,
i.e.\ $\neg x_1 \wedge \neg x_2 \wedge \cdots \wedge \neg x_m$.

\begin{remark}[Relationship between the OR and NOR gates]
The NOR gate is the complement of the OR gate applied to the same set of
variables reinterpreted as repressors:
$\mathrm{NOR}(x_1,\ldots,x_m) = \neg(x_1\vee\cdots\vee x_m) =
\prod_{j=1}^m f^-(x_j)$. Consequently, every structural advantage that the
De~Morgan product OR gate holds over the Samuilik weighted-sum OR gate is
\emph{inherited by and compounded in} the NOR gate, since repression
compounds the pathologies already present in the OR case.
\end{remark}

In the Samuilik framework~\citep{samuilik2022mathematical}, repressors enter
through \emph{negative} weights $w_{ij}<0$ inside the single shared
increasing logistic:
\begin{equation}
\Phi_{\mathrm{WS}}^{\mathrm{NOR}}(x_1,\ldots,x_m)
  = \frac{1}{1+\exp\!\Bigl(-\mu_i\!\Bigl(
      \sum_{j=1}^m w_{ij}\,x_j - \theta_i\Bigr)\Bigr)},
  \quad w_{ij} < 0,
\label{eq:ws_nor}
\end{equation}
with the standard threshold
prescription~\citep{samuilik2022mathematical,kozlovska2022models}
\begin{equation}
  \theta_i = \frac{\sum_{j=1}^m w_{ij}}{2} < 0.
\label{eq:ws_nor_threshold}
\end{equation}

In our framework, all weights are strictly positive ($w_{ij}>0$) and each
repressor~$x_j$ enters through its own \emph{decreasing} logistic:
\begin{equation}
\Phi_{\mathrm{DM}}^{\mathrm{NOR}}(x_1,\ldots,x_m)
  = \prod_{j=1}^{m} \frac{1}{1+e^{\,\lambda(w_{ij}x_j - \theta_{ij})}},
\label{eq:dm_nor}
\end{equation}
where each factor $f^-(w_{ij}x_j,\theta_{ij},\lambda)\in(0,1)$ represents
the uninhibited fraction for repressor~$j$. The complete dynamical system
is $\dot{x}_i = \kappa_i\,\Phi(\mathbf{x}) - \gamma_i x_i$.

We note here that the Samuilik NOR gate inherits the threshold-scaling
pathology of the OR gate and additionally acquires a \emph{sign-flip
pathology}: the prescribed shared threshold
$\theta_i = \sum_j w_{ij}/2 < 0$ is strictly negative and diverges to
$-\infty$ as $m$ grows. This is the same structural defect identified in
Section~\ref{subsec:and_comparison} for the single-repressor Samuilik
function (cf.\ Figure~\ref{fig:comparison}), now extending to the
multi-repressor setting and compounded by network-size scaling. The
weighted De~Morgan product NOR avoids both pathologies: each repressor's
effective threshold $\theta_{ij}/w_{ij}>0$ is positive,
context-independent, and experimentally measurable irrespective of how
many other repressors exist in the network.

\begin{remark}[Alternative unified weighted formulation]
A unified weighted logistic formulation encoding both activation and repression within a single logistic function---with positive weights for activation and negative weights for repression---is also possible, and can in principle offer analytical advantages for high-dimensional systems by reducing the number of nonlinear terms. Achieving functional equivalence with the product-of-logistics formulation~\eqref{eq:multi_gene_system} requires introducing bias-correction terms that depend on the number of regulators per gene; a detailed treatment of this unified framework, including rigorous proofs of equivalence, derivation of the bias-correction terms, and applications to control design, will appear in a forthcoming paper.
\end{remark}


\section{Control Advantages of Logistic Functions in Gene Regulatory Networks}
\label{sec:control}

\subsection{Always-Positive Production and Decoupled Parameters}
\label{subsec:control_basal}

The logistic function provides fundamental advantages over Hill functions for
controlling biological networks, stemming from its non-zero response at minimal
expression levels and analytical tractability. The logistic function
$f^+(x,\theta,\lambda) = 1/(1 + e^{-\lambda(x - \theta)})$ maintains
$f^+(0,\theta,\lambda) = 1/(1 + e^{\lambda\theta}) > 0$, ensuring continuous
regulatory control even when gene expression drops to zero, whereas the Hill
function $h^+(x,\theta,n) = x^n/(x^n + \theta^n)$ vanishes identically at
$x = 0$, creating controllability gaps that compromise feedback regulation.
Biologically, genes exhibit persistent non-zero basal (leaky) expression even
without activators, stabilising low-expression states across multiple cell
divisions as observed in the GAL network in yeast and promoter leakage in
auto-regulatory circuits~\citep{acar2005enhancement,huang2015effects}. Without
repressor proteins, transcription proceeds at sustained high rates as seen in
bacterial operons (e.g., the \textit{lac} system) and phage
lambda~\citep{jacob1961genetic,ackers1982quantitative}, and the logistic
repression function $f^-(0,\theta,\lambda) = 1/(1 + e^{-\lambda\theta})$
provides tunable control over baseline expression via the product
$\lambda\theta$, whereas Hill functions rigidly fix $h^-(0) = 1$, offering no
intrinsic parameter-based modulation. Similarly, the activation function
$f^+(0,\theta,\lambda) = 1/(1 + e^{\lambda\theta})$ provides tunable control of
the basal expression rate via the same product $\lambda\theta$.

The parameters $\lambda$ (steepness) and $\theta$ (threshold) map directly to
tunable molecular properties: $\theta$ represents the dissociation constant
for binding, adjustable through operator mutations or protein
engineering~\citep{zuo2014high,falcon2000operator}, while $\lambda$ governs
cooperativity, modifiable via multimeric repressors or auxiliary binding
sites~\citep{santillan2008use}. These parameters are experimentally accessible
through synthetic biology techniques such as promoter libraries, directed
evolution, and optogenetics~\citep{kumar2023diya,razo2018tuning}. The full
parameter independence inherent in the logistic formulation, where threshold
position $\theta$ and transition steepness $\lambda$ are decoupled, proves
particularly advantageous for control design, enabling independent tuning of
the decision threshold and response sensitivity without compensatory parameter
adjustments. This contrasts sharply with Hill functions, where the maximum
slope $n/(4\theta)$ couples both parameters, requiring simultaneous
readjustment of cooperativity and threshold to maintain desired control
characteristics.

\subsection{Control Strategies Enabled by the Logistic Structure}
\label{subsec:control_strategies}

The structural properties identified above support several distinct control strategies. \emph{Multiplicative control} modulates production rates via control inputs $u_i \geq 0$ in the form
\[
\dot{x}_i = \kappa_i \,
  \frac{u_i}{1 + e^{-\sigma_i\lambda(x_j - \theta_i)}}
  - \gamma_i x_i,
\]
ensuring non-zero controllability at all expression levels. \emph{Steepness modulation} adjusts $\lambda$ through control inputs in
\[
\dot{x}_i = \kappa_i \, f_i\!\left(x_1, \ldots, x_N; u_i\lambda, \theta_{ij}\right)
  - \gamma_i x_i,
\]
yielding linear, predictable control over regulatory sensitivity without the numerical instabilities associated with large Hill coefficients. \emph{State-feedback control} relies on additive corrections $\dot{x}_i = \kappa_i f_i(\mathbf{x}) - \gamma_i x_i + u_i$ with $u_i = -K_i(x_i - x_{d,i})$, leveraging continuous responsiveness to achieve exponential convergence to desired setpoints. \emph{Sliding-mode control} benefits from the logistic function's smooth, bounded character, ensuring robust performance under parameter uncertainties~\citep{belgacem2020control,chambon2020qualitative,sawlekar2015biomolecular}, whereas Hill-based models suffer fragility in equivalent control laws because of zero production at $x = 0$, undefined expressions, and fractional exponents. Finally, \emph{model predictive control} exploits the closed-form logit inverse
\[
f^{-1}(y) = \theta + \frac{1}{\lambda}\ln\!\left(\frac{y}{1-y}\right),
\]
which enables exact feedback linearisation and gradient-based optimisation with smooth derivatives, yielding well-conditioned optimisation problems compared with Hill functions' power-law nonlinearities~\citep{del2015biomolecular,lugagne2024deep,faquir2025computational}.

To make the feedback-linearisation claim concrete, consider a single autoregulated gene with controlled production rate,
\begin{equation}
\dot{x} \;=\; u(t)\,\kappa\,f^{+}(x,\theta,\lambda) \;-\; \gamma x,
\qquad u(t)\in[u_{\min},u_{\max}],\;u_{\min}>0,
\label{eq:scalar_logistic_control}
\end{equation}
where $u(t)$ is a multiplicative control input (e.g.\ an inducer concentration or a light intensity in optogenetic platforms~\citep{lugagne2024deep}).  Suppose we wish the closed-loop dynamics to track a desired reference $y_{\mathrm{ref}}(t)\in(0,1)$ in the \emph{normalised output} $y=f^{+}(x,\theta,\lambda)$.  Differentiating $y$ along trajectories of~\eqref{eq:scalar_logistic_control},
\[
\dot{y} \;=\; \lambda y(1-y)\,\dot{x}
       \;=\; \lambda y(1-y)\bigl(u\kappa y - \gamma x\bigr),
\qquad x \;=\; \theta + \tfrac{1}{\lambda}\ln\!\tfrac{y}{1-y}.
\]
The feedback law
\begin{equation}
u(y) \;=\; \frac{1}{\kappa y}\!\left[
   \gamma\!\left(\theta + \tfrac{1}{\lambda}\ln\!\tfrac{y}{1-y}\right)
   + \frac{v(t)}{\lambda y(1-y)}
\right]
\label{eq:feedback_linearisation}
\end{equation}
exactly cancels the nonlinearity and reduces the closed-loop output dynamics to the integrator $\dot{y}=v(t)$, with $v$ free to be designed by classical linear methods (PI control, LQR, MPC).  Crucially, every quantity in~\eqref{eq:feedback_linearisation} is well-defined and analytic on the open interval $y\in(0,1)$; the closed-form logit inverse $\theta+\lambda^{-1}\ln\!\tfrac{y}{1-y}$ uses only $\exp$ and $\log$, whereas the Hill inverse $\theta(y/(1-y))^{1/n}$ carries a fractional power that is non-smooth at the origin for non-integer $n$.  The control input $u$ is bounded on every closed subinterval $y\in[\eta,1-\eta]\subset(0,1)$: the inverse $\theta+\lambda^{-1}\ln\!\tfrac{y}{1-y}$ and the gain factor $1/\bigl(\lambda y(1-y)\bigr)$ are both bounded there, so the actuator range required for tracking is finite---set by $\eta$ and by the bound on $\dot{y}_{\mathrm{ref}}$ along the reference. This is the standard tracking domain of feedback linearisation and is enforced in practice by the saturation bounds $[u_{\min},u_{\max}]$ on the actuator.  The same construction extends straightforwardly to the multi-input setting~\eqref{eq:multi_gene_system} via the product structure of $\Phi_i$.

The qualitative tracking claim can be made precise as an explicit
exponential decay of the error.

\begin{proposition}[Exact exponential tracking via feedback linearisation]
\label{prop:tracking}
Let $\kappa,\gamma,\lambda,\theta>0$ and $K>0$, and consider the
controlled scalar system~\eqref{eq:scalar_logistic_control}.  Let
$y_{\mathrm{ref}}\in C^{1}([0,\infty);(0,1))$ be a reference trajectory
that takes values in a fixed compact subinterval $[\eta,1-\eta]\subset(0,1)$
for some $\eta\in(0,1/2)$.  Apply the feedback law~\eqref{eq:feedback_linearisation}
with auxiliary input
\[
v(t) \;=\; \dot{y}_{\mathrm{ref}}(t) - K\bigl(y(t)-y_{\mathrm{ref}}(t)\bigr).
\]
Provided the resulting actuator command $u(y(t))\in[u_{\min},u_{\max}]$
for all $t\ge 0$ (which holds whenever $y(t)$ stays in the tracking domain
$[\eta,1-\eta]$, by the bound stated above), the tracking error
$e(t):=y(t)-y_{\mathrm{ref}}(t)$ satisfies the linear closed-loop ODE
\[
\dot{e}(t)\;=\;-K\,e(t),
\qquad
e(t)\;=\;e(0)\,e^{-Kt}\quad\forall t\ge 0,
\]
so $|e(t)|=|e(0)|\,e^{-Kt}\to 0$ at the chosen exponential rate $K$.
\end{proposition}

\begin{proof}
By construction, the feedback law~\eqref{eq:feedback_linearisation} reduces
the closed-loop output to $\dot{y}=v(t)$.  Substituting the choice of $v$,
$\dot{y} = \dot{y}_{\mathrm{ref}} - K(y-y_{\mathrm{ref}})$, hence
$\dot{e} = \dot{y} - \dot{y}_{\mathrm{ref}} = -Ke$.  Integration yields
$e(t)=e(0)e^{-Kt}$.
\end{proof}

Three remarks are in order. First, the rate $K$ is a free design
parameter; larger $K$ gives faster tracking at the cost of a larger
control authority $|v|$, hence a larger required range of $u$. Second,
because every $y$-dependent factor in the feedback law is bounded on
$[\eta,1-\eta]$, the actuator dynamic range needed for tracking on that fixed
subinterval is finite and is governed by $\eta$ and by the bound on
$\dot{y}_{\mathrm{ref}}$, not by the amplitude of the reference itself.  Third,
the corresponding
construction for the Hill function rests on the fractional-power inverse
$\theta(y/(1-y))^{1/n}$; this is closed form, but its root---unlike the
logit---is non-smooth at the origin and turns complex for any negative
excursion, so the logistic yields the cleaner actuation law.

\subsection{Multi-Dimensional and Combinatorial Control}
\label{subsec:control_multi}

For multi-dimensional systems with cooperative regulation combining activator
and repressor effects, the controlled dynamics for parallel regulation become
\[
\dot{x}_i = \kappa_i
\!\left(
  \prod_{j\in\mathcal{A}_i}
    \frac{1}{1 + e^{-u_{ij}\lambda(x_j - \theta_{ij})}}
  \;\cdot\;
  \prod_{k\in\mathcal{R}_i}
    \frac{1}{1 + e^{-u_{ik}\lambda(\theta_{ik} - x_k)}}
\right)
- \gamma_i x_i,
\]
where $u_{ij}, u_{ik} \geq 0$ modulate the steepness of each regulatory
interaction, ensuring non-zero production at zero regulator concentrations,
unlike Hill-based models where the regulatory term
\[
f_i(x_j, x_k) =
  \frac{x_j^n}{x_j^n + \theta_{ij}^n} \cdot
  \frac{\theta_{ik}^n}{x_k^n + \theta_{ik}^n}
\]
vanishes identically whenever the activator concentration $x_j$ reaches zero
(since $h^+(0) = 0$), rendering systems uncontrollable in activator-absent states.

An alternative control strategy modulates regulatory influences directly
through the control matrix elements $u_{ij}$ and $u_{ik}$,
\[
\dot{x}_i = \kappa_i
\!\left(
  \prod_{j\in\mathcal{A}_i}
    \frac{1}{1 + e^{-\lambda(u_{ij}x_j - \theta_{ij})}}
  \;\cdot\;
  \prod_{k\in\mathcal{R}_i}
    \frac{1}{1 + e^{-\lambda(\theta_{ik} - u_{ik}x_k)}}
\right)
- \gamma_i x_i,
\]
with $u_{ij}, u_{ik} \geq 0$,
enabling dynamic modulation of regulatory influences for targeted interventions
in optogenetic applications.

\subsection{Linearisation, Controllability, and Accessibility}
\label{subsec:control_linear}

In contrast to Hill functions, the logistic model admits tractable linear
approximations both near the origin and near the inflection point (threshold)
when $\lambda$ is small, yielding an analytically tractable linearised system
$\dot{\mathbf{x}} = A\mathbf{x} + \mathbf{b} + B\mathbf{u}$. This facilitates
controllability analysis via rank conditions on the Kalman matrix
$\mathcal{C} = \begin{bmatrix} B & AB & \cdots & A^{N-1}B \end{bmatrix}$,
the application of linear control tools (pole placement, LQR,
$\mathcal{H}_\infty$), and systematic gain selection for desired convergence
and robustness. The bilinear structure that emerges in controlled systems
enables accessibility analysis through Lie-algebra methods.

Compared with Hill functions---which exhibit zero production at activator
absence (rendering systems uncontrollable in low-expression regimes),
non-smooth behaviour for large $n$ (causing numerical instability), complex
rational expressions (complicating SMC and MPC design), and lack of
closed-form derivatives for non-integer $n$ (hindering optimisation)---the
logistic functions provide always-positive production maintaining
controllability, smooth bounded responses ensuring numerical stability,
closed-form derivatives and inverses facilitating analytical control design,
and parameters that map directly onto tunable biological mechanisms.

\subsection{Practical Considerations and Experimental Implementation}
\label{subsec:control_practical}

In practical sliding-mode implementations, quasi-sliding-mode controllers using DNA strand-displacement reactions have been demonstrated by \citet{sawlekar2015biomolecular}, outperforming traditional linear controllers with faster tracking response and no overshoot, both critical for genetic networks requiring precise control. Chattering---the rapid switching caused by discontinuous sign functions---is mitigated through the introduction of a boundary layer, replacing $\operatorname{sign}(s_i)$ with the smooth saturation
\[
\operatorname{sat}(s_i, \epsilon)
= \begin{cases}
    \operatorname{sign}(s_i) & \text{if } |s_i| > \epsilon, \\
    s_i / \epsilon          & \text{if } |s_i| \leq \epsilon,
  \end{cases}
\]
where $\epsilon > 0$ defines the boundary-layer thickness, ensuring smooth control transitions within $|s_i| \leq \epsilon$ while maintaining robust reaching behaviour outside this region.

For multi-gene networks such as the repressilator with cyclic inhibitory interactions, the logistic framework's analytical tractability enables systematic design of feedback controllers that stabilise oscillations at desired amplitudes or frequencies, synchronise multiple circuits, track time-varying trajectories, and compensate for cell-to-cell variability~\citep{belgacem2020control,chambon2020qualitative}. Linear approximations yield cyclic coupling structures analysable via circulant matrix theory, while MPC strategies exploit predictive capability for phase-locking and for maintaining oscillation characteristics under disturbances. These control strategies are experimentally feasible through optogenetics (which provides millisecond-precision control inputs), fluorescent reporters (which enable real-time expression measurements to close feedback loops), and microfluidic platforms (which permit parallel control of thousands of cells for population studies); the Khammash laboratory at ETH Z\"urich has pioneered such experimental platforms~\citep{kumar2023diya}.

A detailed treatment of the control strategies sketched above is left for future work. Promising directions include extending these strategies to stochastic gene networks accounting for intrinsic and extrinsic noise in low-copy-number regimes; integrating spatial dynamics and cell-to-cell communication in multicellular systems requiring distributed control architectures; combining logistic-based mechanistic models with machine-learning components for adaptive control in uncertain or time-varying environments; experimentally validating these control strategies in optogenetic systems with millisecond-precision real-time feedback; and exploring applications beyond gene regulation, including metabolic pathway control, cell-cycle regulation, and morphogen-gradient formation in developmental biology.

In summary, logistic-based models provide the mathematical tractability, biological fidelity, and computational robustness needed to control gene regulatory networks. Their non-zero basal activity ensures continuous controllability---essential for feedback architectures that must respond to weak inputs---while their smooth, analytically tractable form enables the systematic application of advanced control methods ranging from sliding-mode to model-predictive control, with direct experimental implementation pathways through modern synthetic biology.

\section{Conclusion}
\label{sec:conclusion}

This paper has developed a complete product-of-logistics framework for
modelling gene regulatory networks and established its core analytical
properties. The framework deploys increasing logistic functions for activation
and decreasing logistic functions for repression, each precisely where it is
biologically appropriate, and thereby preserves the distinct sigmoidal dynamics
of the two regulatory modes while inheriting the structural advantages of the
logistic form.

These advantages are not incidental. The global $C^\infty$ regularity of
logistic functions removes the origin singularities that afflict Hill functions
whenever the cooperativity exponent is non-integer. The self-referential
derivative identity $f' = \lambda f(1-f)$ reduces Jacobian entries to products
of function values, eliminating the fractional-power evaluations that make
Hill-function Jacobians expensive and ill-conditioned near zero. The closed-form
logit inverse $f^{-1}(y) = \theta + \lambda^{-1}\ln(y/(1-y))$ supports exact
threshold calculations and feedback linearisation. The full decoupling of
steepness $\lambda$ and threshold $\theta$---whose maximum slope $\lambda/4$
depends on $\lambda$ alone, in contrast to the Hill maximum slope $n/(4\theta)$,
which entangles both design variables---allows independent tuning of the
decision threshold and the response sensitivity.

For the two-gene negative-feedback oscillator, Jacobian analysis and the
Routh--Hurwitz criterion establish local asymptotic stability at the unique
equilibrium for all biologically meaningful parameter values. Combined with
Bendixson's negative criterion---which rules out closed orbits using the
constant negative divergence $-(\gamma_1+\gamma_2)$ throughout
$\mathbb{R}^{2}$---and the Poincar\'e--Bendixson theorem applied to the
forward-invariant box, this upgrades to \emph{global} asymptotic stability on
$\mathbb{R}^{2}$ (Theorem~\ref{thm:no_hopf}). A key structural consequence is
that the two-dimensional logistic oscillator cannot undergo a Hopf bifurcation:
the trace of the Jacobian cannot be made to vanish by any choice of $\lambda$,
$\kappa_i$, or $\theta_i$, so sustained limit cycles require the introduction of
time delays~\citep{belgacem2025glass,farcot2019chaos}.

The complementary two-gene motif---mutual repression---was analysed through
the genetic toggle switch of~\citet{gardner2000construction}.
Proposition~\ref{prop:toggle} establishes that the logistic toggle has purely
real Jacobian eigenvalues and a negative trace at every equilibrium, hence
cannot oscillate, and that the determinant identity
$\det J=\gamma_1\gamma_2\bigl(1-T'(x_1^{*})\bigr)$ renders the switch
monostable or bistable according to whether a scalar return map crosses the
diagonal with slope below or above unity; in the symmetric case the
monostable--bistable transition is a supercritical pitchfork at
$\rho\lambda=4$. A fourth-order Runge--Kutta simulation in R confirms the
bistable phase portrait and this threshold.

At the network scale, the general multi-gene system was formulated as a
product-of-logistics map. Proposition~\ref{prop:demorgan} establishes the three
structural properties of the recursive De~Morgan map---range, Boolean
consistency under uniform convergence, and De~Morgan duality---and
Proposition~\ref{prop:wellposedness} establishes that the resulting multi-gene
ODE is globally well-posed, forward-invariant on the box
$\prod_i[0,\kappa_i/\gamma_i]$, and globally Lipschitz with the explicit
constant
$L=\max_i\bigl(\gamma_i+\kappa_i\lambda(|\mathcal{A}_i|+|\mathcal{R}_i|)/4\bigr)$,
which is bounded uniformly in the network size whenever in-degree and $\lambda$
are bounded. Theorem~\ref{thm:boolean_recovery} closes the loop between the
continuous and discrete descriptions: under the natural threshold condition
$0<\theta_i<\kappa_i/\gamma_i$, every steady state of the Boolean network is
recovered, for sufficiently steep regulatory response, as an exponentially
stable equilibrium of the continuous model, so that the translation provably
refines rather than distorts the original Boolean analysis. The general
$m$-clause recursive product formula
$\Phi\!\left(\bigvee_{k=1}^m C_k\right) = 1 - \prod_{k=1}^m\bigl(1 -
\Phi(C_k)\bigr)$ does not appear to have been stated explicitly in the prior
literature on continuous approximations of Boolean GRNs;
\citet{wittmann2009transforming} recover the algebraically equivalent two-input
instance by polynomial interpolation but do not identify it as a product of
decreasing logistic functions nor state the general recursive form. As a
concrete application, the $11$-gene Traynard mammalian cell-cycle Boolean
network~\citep{Traynard2016}---governing Cdc20, Cdh1, CycA, CycB, CycD, CycE,
E2F, p27, Rb, Skp2, and UbcH10---was translated automatically into a continuous
ODE system; the resulting $11$-dimensional system integrates without warnings,
all state variables remain non-negative and---after the initial transient in
which the two components launched above their ceiling relax into it---bounded by
$\kappa_i/\gamma_i$, and in the proliferative regime the trajectories settle
onto a sustained limit cycle that reproduces the cyclic attractor of the Boolean
network.

The framework also accommodates explicit interaction weights.
Section~\ref{sec:weight_equivalence} establishes formally that incorporating
real-valued weights into the product-of-logistics formulation is equivalent to
the fixed-weight formulation after the parameter rescaling
$\lambda'_{ij} = \lambda_{ij} w_{ij}$ and $\theta'_{ij} =
\theta_{ij}/w_{ij}$; the two parameterisations produce identical dynamics and
differ only in how steepness and threshold information is distributed across
parameters.

Finally, the detailed comparison with the Samuilik weighted-sum
formulation~\citep{samuilik2022mathematical} across the AND, OR, and NOR logic
gates shows that the product-of-logistics formulation offers biological
interpretability that single-sigmoid, shared-threshold alternatives cannot
match. Each threshold $\theta_{ij}$ is directly interpretable as a dissociation
constant determinable from independent experimental measurements, enabling
decomposed validation; the product structure encodes AND logic transparently;
and, unlike the single increasing sigmoid with signed weights, it does not
place repressor critical points at biologically meaningless negative
concentrations. On the repression side, the decreasing logistic function
naturally approaches---but need not exactly reach---unity in the absence of
repressor, capturing polymerase saturation, resource competition, and
stochastic promoter switching; for applications requiring exact normalisation,
the scaled variant $f^-_{\mathrm{scaled}}(x,\theta,\lambda) =
(1+e^{-\lambda\theta})\,f^-(x,\theta,\lambda)$ restores unit value at zero
repressor concentration, with a scaling factor that is negligible under typical
parameter regimes ($\lambda\theta \geq 4$).

The final contribution of this paper concerns control. The always-positive
logistic production rate eliminates the controllability gaps that Hill
functions create at zero concentration, where zero production precludes
recovery through intrinsic dynamics alone; multiplicative control, steepness
modulation, sliding mode control, and model predictive control are all directly
applicable at expression levels where Hill-based formulations lose
controllability. The closed-form logit inverse enables an explicit
feedback-linearisation construction that exactly cancels the sigmoid
nonlinearity; Proposition~\ref{prop:tracking} formalises this as exact
exponential tracking $|y(t)-y_{\mathrm{ref}}(t)| =
|y(0)-y_{\mathrm{ref}}(0)|\,e^{-Kt}$ at a chosen design rate $K$, on any
reference confined to a compact subinterval of $(0,1)$. The full decoupling of
threshold $\theta$ and steepness $\lambda$ permits independent tuning of the
decision threshold and the response sensitivity.

The framework developed here is the foundation for the companion
paper~\citep{belgacem2026numerical}, which examines the biological realism of
the logistic choice at low expression levels and the numerical reliability of
Boolean-derived ODE integration relative to Hill functions with non-integer
exponents. By replacing Hill functions with their logistic
counterparts while preserving sigmoidal dynamics, researchers can build on
decades of accumulated Hill-based modelling intuition while gaining the
analytical tractability that demanding applications in systems and synthetic
biology require.

\section*{Statements and Declarations}

\subsection*{Competing interests}
The author declares that he has no competing financial interests or
personal relationships that could have appeared to influence the work
reported in this paper.

\subsection*{Funding}
This research received no specific grant from any funding agency in the
public, commercial, or not-for-profit sectors.

\subsection*{Data availability}
The genetic oscillator simulations of Section~\ref{ex:oscillator}
(Figure~\ref{fig:Oscillateur_original}) were conducted in R using the
\texttt{deSolve} package's \texttt{ode} function, with parameter values
$\lambda = 3$, $\kappa_1 = 3$, $\gamma_1 = 0.25$, $\kappa_2 = 4$,
$\gamma_2 = 0.5$, $\theta_1 = 4$, $\theta_2 = 3$, and initial conditions
$x_1(0) = x_2(0) = 1$. The genetic toggle-switch simulation of
Section~\ref{ex:toggle} (Figure~\ref{fig:toggle}) was performed in R with a
fourth-order Runge--Kutta integrator, using the symmetric parameters
$\kappa=10$, $\gamma=1$, $\theta=5$. The Traynard cell-cycle ODE simulation
of Section~\ref{ex:traynard} (Figure~\ref{fig:traynard_sim}) was implemented
in \textit{Mathematica}. No experimental datasets were generated; all parameter
values are drawn from published literature
in~\citep{madar2011negative,oehler1994quality,kauffman1969metabolic,%
albert2003topology,ingalls2013mathematical}.

\bibliographystyle{elsarticle-num-names}
\bibliography{mybibfile_corrected}

\end{document}